\documentclass[12pt,a4paper]{article}
\usepackage[latin1]{inputenc}
\usepackage{amsmath, amsthm}
\usepackage{amsfonts}
\usepackage{amssymb}
\usepackage{makeidx}
\usepackage{graphicx}
\usepackage{multirow}
\newtheorem{definition}{Definition}[section]
\newtheorem{proposition}{Proposition}[section]
\newtheorem{theorem}{Theorem}[section]
\newtheorem{remark}{Remark}[section]
\newtheorem{lemma}{Lemma}[section]
\newtheorem{corollary}{Corollary}[section]
\newtheorem{example}{Example}[section]
\usepackage[width=180.00mm, height=195.00mm]{geometry}
\usepackage{color}
\usepackage{pdflscape}
\usepackage{longtable, float}

\usepackage{hyperref}

\author{Ratan Lal, Alka Choudhary and Vipul Kakkar}
\title{Determinant for automorphisms of Zappa-Sz\'{e}p product of groups}
\allowdisplaybreaks
\date{}
\begin{document}

	\maketitle
	\begin{flushleft}
		\textrm{Desh Bhagat Pandit Chetan Dev Government College of Education, Faridkot, Punjab, India}\\
	\textrm{Manipal University, Jaipur, Rajasthan, India}\\
	\textrm{Central University of Rajasthan, Ajmer, Rajasthan, India}\\
\end{flushleft}
	E-mails: \url{vermarattan789@gmail.com}, \url{alkababal@gmail.com}, \url{vplkakkar@gmail.com}
	
	\abstract{A description of the endomorphisms of Zappa-Sz\'{e}p products of two groups as a group of $2\times 2$ matrices of maps is discussed in \cite{detsem}. Using this description, we have studied the concept of determinant for the endomorphisms of Zappa-Sz\'{e}p products of two groups. A characterization of the invertible endomorphisms is given with the help of the tools developed using the determinants. The determinants are also computed for the alternating group $A_{5}$. }\\
		

	\section{Introduction}
	Factorization theory plays a central role in structural group theory, offering a systematic framework to decompose complex groups into simpler constituent factors. When a group $G$ factors as a product of two subgroups $H$ and $K$ such that $G = HK$ and $H \cap K = \{1\}$, the algebraic structure of $G$ is entirely dictated by the mutual interactions between $H$ and $K$. If either of the factors is a normal subgroup, then we get direct product or semidirect product. When neither factor is normal, $G$ generalizes to the \emph{Zappa--Szép product} (also known as the mutual semidirect product or knit product).
	
	A group $G$ is called the internal Zappa-Sz\'{e}p product of its two subgroups $H$ and $K$ if $G=HK$ and $H\cap K = \{1\}$. If $G$ is the internal Zappa-Sz\'{e}p product of $H$ and $K$, then $K$ appears as a right transversal to $H$ in $G$. Let $h\in H$ and $k\in K$. Then $kh=\sigma(k,h)\tau(k,h)$, where $\sigma(k,h)\in H$ and $\tau(k,h)\in K$. This determines the maps $\sigma: K \times H \rightarrow H$ and $\tau: K\times H \rightarrow K$ defined by $\sigma(k,h) = \sigma_{k}(h)$ and $\tau(k,h) = \tau_{h}(k)$. These maps are called the matched pair of groups or the mutual actions of $H$ and $K$. We denote these maps by $\sigma(k, h) = k\cdot h$ and $\tau(k, h) = k^{h}$ for all $h\in H$ and $k\in K$. These maps satisfy the following conditions for all $h,h^{\prime} \in H$ and $k,k^{\prime}\in K$ (see \cite{af}):	
	\begin{itemize}
		\item[$(C1)$] $1\cdot h = h$ and $k^{1} = k$,
		\item[($C2$)] $k\cdot 1 = 1 = 1^{h}$,
		\item[$(C3)$] $kk^{\prime}\cdot h = k\cdot (k^{\prime}\cdot h)$,
		\item[($C4$)] $(kk^{\prime})^{h} = k^{k^{\prime}\cdot h}{k^{\prime}}^{h}$,
		\item[($C5$)] $k\cdot (hh^{\prime}) = (k\cdot h)(k^{h}\cdot h^{\prime})$,
		\item[$(C6)$] $k^{hh^{\prime}} = (k^{h})^{h^{\prime}}$,
	\end{itemize} 

	On the other hand, let $H$ and $K$ be two groups and define two maps \[\sigma: K\times H\longrightarrow H\; \text{and}\; \tau: K\times H \longrightarrow K\]
	by $\sigma(k, h) = k\cdot h$ and $\tau(k, h) = k^{h}$ satisfying the conditions $(C1)-(C6)$ for all $h\in H$ and $k\in K$. Then the set $H\times K$ under the binary operation defined for all $(h, k), (h^{\prime}, k^{\prime})\in H\times K$ as
	\begin{equation*}
	(h, k)(h^{\prime}, k^{\prime}) = (h(k\cdot h^{\prime}), k^{h^{\prime}}k)
	\end{equation*}
	forms a group called the external Zappa-Sz\'{e}p product of the groups $H$ and $K$ and is denoted by $H\bowtie K$. By (\cite[Proposition 2.4, p. 4]{af}), the internal Zappa-Sz\'{e}p product is isomorphic to the external Zappa-Sz\'{e}p product. Therefore, we will identify the internal Zappa-Sz\'{e}p product  with the external Zappa-Sz\'{e}p product.
		
The study of automorphism groups of groups have attracted many algebraists over the years. The automorphism groups of direct products \cite{bcm, jb}, semidirect products \cite{cur, jd, hsu} and Zappa-Sz\'{e}p products \cite{zapfix, zap} have been studied as the group of $2\times 2$ matrices of maps that satisfy some specific conditions. 

In \cite{mb}, the author put a step forward and studied the matrix characterizations of the automorphism groups and given some algorithms for an endomorphism of direct product of groups to be invertible. For an endomorphism $\phi$ of direct product of two groups $H$ and $K$ represented by $\begin{pmatrix}
\alpha & \beta \\ \gamma & \delta
\end{pmatrix}$, the author in \cite{mb} have defined the $H$--determinant as $\det_{H}(\phi) = \alpha - \beta\delta^{-1}\gamma$ (when $\delta$ is an invertible map) and $K$--determinant as $\det_{K}(\phi) = \delta - \gamma\alpha^{-1}\beta$ (when $\alpha$ is an invertible map). Both of these determinants yields an effective criterion for invertibility of the endomorphism $\phi$.  A similar theory have recently been explored for the semidirect products of two groups \cite{detsem}.

Since Zappa-Sz\'{e}p product of two groups is a natural generalization of the semidirect product of two groups, a natural question arises: \emph{Can the determinant framework for the endomorphisms of direct products and semidirect products be extended to construct a determinant framework for endomorphisms of the Zappa-Sz\'{e}p product of two groups?} 

This paper provides an affirmative answer to the above question. The main findings and the structure of the paper are organized as follows:
\begin{itemize}
	\item In \textbf{Section \ref{s2}}, a monoid isomorphism between the endomorphism monoid $\mathrm{End}(H \bowtie K)$ and a monoid consisting of $2 \times 2$ matrices of maps satisfying some conditions is established (Theorem \ref{t1}).
	\item In \textbf{Section \ref{s3}}, a precise criteria for an endomorphism of $H \bowtie K$ to be an automorphism is established (Theorem \ref{t32} and \ref{t33}). Furthermore, it is proved that $\det_{H}$ is invertible if and only if $\det_{K}$ is invertible (Theorem \ref{t34}).
	\item In \textbf{Section \ref{s4}}, we demonstrate the theoretical framework by explicitly computing both the determinants $\det_{H}$ and $\det_{K}$ for all the endomorphisms of the alternating group $A_{5}$.
\end{itemize}


	\paragraph{Notation and Conventions.} Let $H$ and $K$ be two groups. Then throughout this paper, $\mathrm{Hom}(H, K)$, $\mathrm{End}(H)$, and $\mathrm{CHom}(H, K)$ denote the group of homomorphisms from $H$ to $K$, the monoid of endomorphisms of $H$, and the set of crossed homomorphisms from $H$ to $K$, respectively. $Z(H)$ denotes the center of $H$, $I_{H}$ denotes the identity map on $H$, and $\gcd$ denotes the greatest common divisor.
	
	
	\section{Preliminaries}\label{s2}
	Let $U$, $V$ and $W$ be any groups. Let $Map(U, V)$ denote the set of all maps from $U$ to $V$. If $\phi, \psi \in Map(U, V)$ and $\eta \in Map(V, W)$, then $\phi + \psi \in	Map(U, V)$ is defined by $(\phi + \psi)(u) = \phi(u)\psi(u)$, $\eta \phi\in Map(U, W)$ is defined by $\eta\phi(u) = \eta(\phi(u))$, $-\phi \in Map(U, V)$ is defined as $-\phi(u) = \phi(u)^{-1}$ and $\phi \cdot \psi$, $\phi^{\psi}\in Map(U, V)$ are defined by $(\phi \cdot \psi)(u) = \phi(u)\cdot \psi(u)$, $(\phi^{\psi})(u) = \phi(u)^{\psi(u)}$ for all $u\in U$.

Let $H$ and $K$ be groups and let
\begin{equation*}
\tilde{\mathcal{M}} = \left\{\begin{pmatrix}
\alpha & \beta\\ \gamma & \delta
\end{pmatrix}
\mid \begin{matrix}
\alpha \in Map(H, H), & \beta\in Map(K, H)\\ \gamma\in Map(H, K), & \text{and}\;\delta \in Map(K, K)
\end{matrix} \right\}.
\end{equation*}

We define $\mathcal{M}$ to be the subset of $\tilde{\mathcal{M}}$ consisting of all matrices
\begin{align*}
\begin{pmatrix}
\alpha & \beta \\ \gamma & \delta
\end{pmatrix}
\end{align*}
such that for all $h,h^{\prime}\in H$ and $k, k^{\prime}\in K$, the following conditions hold:
	\begin{itemize}
		\item[$(i)$] $\alpha(hh^{\prime}) = \alpha(h)(\gamma(h)\cdot \alpha(h^{\prime}))$,
		\item[$(ii)$] $\gamma(hh^{\prime}) = \gamma(h)^{\alpha(h^{\prime})}\gamma(h^{\prime})$,
		\item[$(iii)$] $\beta(kk^{\prime}) = \beta(k)(\delta(k)\cdot \beta(k^{\prime}))$,
		\item[$(iv)$] $\delta(kk^{\prime}) = \delta(k)^{\beta(k^{\prime})}\delta(k^{\prime})$,
		\item[$(v)$] $\beta(k)(\delta(k)\cdot \alpha(h)) = \alpha(k\cdot h)(\gamma(k\cdot h)\cdot \beta(k^{h}))$,
		\item[$(vi)$] $\delta(k)^{\alpha(h)}\gamma(h) = \gamma(k\cdot h)^{\beta(k^{h})}\delta(k^{h})$, 
	\end{itemize}
	 Then the set $\mathcal{M}$ forms a monoid with the binary operation defined as
	\[\begin{pmatrix}
	\alpha^{\prime} & \beta^{\prime}\\ \gamma^{\prime} & \delta^{\prime}
	\end{pmatrix}\begin{pmatrix}
	\alpha & \beta\\ \gamma & \delta
	\end{pmatrix} = \begin{pmatrix}
	\alpha^{\prime}\alpha + (\gamma^{\prime}\alpha \cdot \beta^{\prime}\gamma) & \alpha^{\prime}\beta + (\gamma^{\prime}\beta \cdot \beta^{\prime}\delta)\\ (\gamma^{\prime}\alpha)^{\beta^{\prime}\gamma} + \delta^{\prime}\gamma & (\gamma^{\prime}\beta)^{\beta^{\prime}\delta} + \delta^{\prime}\delta
	\end{pmatrix}  \]
	and the identity element as $\begin{pmatrix}
	1 & 0\\ 0 & 1
	\end{pmatrix}$, where $1$ denotes the identity morphism and $0$ denotes the trivial morphism.

	Let $G = H\bowtie K$. Let $\theta \in End(G)$, $\alpha\in Map(H, H)$, $\beta\in Map(K, H)$, $\gamma\in Map(H, K)$ and $\delta\in Map(K, K)$. Then we say that the maps $\alpha$, $\beta$, $\gamma$ and $\delta$ are associated with the endomorphism $\theta$, if $\theta$ is defined by $\theta(h) = \alpha(h)\gamma(h)$ and $\theta(k) = \beta(k)\delta(k)$ for all $h\in H$ and $k\in K$.
	
		\begin{lemma}\label{lem2.1}
		Let $G = H\bowtie K$ and let $\alpha$, $\beta$, $\gamma$ and let $\delta$ the maps associated with $\theta\in End(G)$. Then $\begin{pmatrix}
		\alpha & \beta \\ \gamma & \delta
		\end{pmatrix}\in \mathcal{M}$.
	\end{lemma}
	\begin{proof}
		 Let $\theta \in End(G)$. Then for all $h\in H$ and $k\in K$, we define the maps $\alpha$, $\beta$, $\gamma$ and $\delta$ by means of $\theta(h) = \alpha(h)\gamma(h)$ and $\theta(k) = \beta(k)\delta(k)$. Now, for all $h,h^{\prime}\in H$, we have
		\begin{align*}
		\alpha(hh^{\prime})\gamma(hh^{\prime}) &= \theta(hh^{\prime})\\ &= \theta(h)\theta(h^{\prime})\\
		&= \alpha(h)\gamma(h)\alpha(h^{\prime})\gamma(h^{\prime})\\
		&= \alpha(h)(\gamma(h)\cdot \alpha(h^{\prime})) \gamma(h)^{\alpha(h^{\prime})}\gamma(h^{\prime}).
		\end{align*}
		Therefore, by the uniqueness of representation, we get $\alpha(hh^{\prime}) = \alpha(h)(\gamma(h)\cdot \alpha(h^{\prime}))$ and $\gamma(hh^{\prime}) = \gamma(h)^{\alpha(h^{\prime})}\gamma(h^{\prime})$. Using a similar argument, we get the conditions $(iii)$ and $(iv)$. Now, for all $k,k^{\prime}\in K$, we have
		\begin{align*}
		\beta(kk^{\prime})\delta(kk^{\prime}) &= \theta(kk^{\prime})\\
		&= \theta(k)\theta(k^{\prime})\\
		&= \beta(k)\delta(k)\beta(k^{\prime})\delta(k^{\prime})\\
		&= \beta(k)(\delta(k)\cdot \beta(k^{\prime}))\delta(k)^{\beta(k^{\prime})}\delta(k^{\prime}).
		\end{align*}
		Using the uniqueness of representation, we get $\beta(kk^{\prime}) = \beta(k)(\delta(k)\cdot \beta(k^{\prime}))$ and $\delta(kk^{\prime}) = \delta(k)^{\beta(k^{\prime})}\delta(k^{\prime})$.
		
		Now, for all $h\in H$ and $k\in K$, we have
		
		\begin{align*}
		\beta(k)(\delta(k)\cdot \alpha(h))\delta(k)^{\alpha(h)}\gamma(h) &= \beta(k)\delta(k)\alpha(h)\gamma(h)\\
		&= \theta(k)\theta(h)\\
		&= \theta(kh)\\
		&= \theta((k\cdot h)k^{h})\\
		&= \theta(k\cdot h)\theta(k^{h})\\
		&= \alpha(k\cdot h)\gamma(k\cdot h)\beta(k^{h})\delta(k^{h})\\
		&= \alpha(k\cdot h)(\gamma(k\cdot h)\cdot \beta(k^{h}))\gamma(k\cdot h)^{\beta(k^{h})}\delta(k^{h}).
		\end{align*}
		Again, using the uniqueness of representation, we get $\beta(k)(\delta(k)\cdot \alpha(h)) = \alpha(k\cdot h)(\gamma(k\cdot h)\cdot \beta(k^{h}))$ and $\delta(k)^{\alpha(h)}\gamma(h) = \gamma(k\cdot h)^{\beta(k^{h})}\delta(k^{h})$. Thus the maps $\alpha$, $\beta$, $\gamma$ and $\delta$ satisfy all the conditions $(i)-(vi)$. Therefore, $\begin{pmatrix}
		\alpha & \beta\\
		\gamma & \delta
		\end{pmatrix} \in \mathcal{M}$.
	\end{proof}

	Now, we prove the isomorphism between the monoids $End(H\bowtie K)$ and $\mathcal{M}$ in the following Theorem.

	\begin{theorem}\label{t1}
		Let $G = H\bowtie K$. Then $End(H\bowtie K) \simeq \mathcal{M}$ as monoids. 
	\end{theorem}
	\begin{proof}
	 Let $\theta \in End(G)$. Then for all $h\in H$ and $k\in K$, we define the maps $\alpha$, $\beta$, $\gamma$ and $\delta$ by means of $\theta(h) = \alpha(h)\gamma(h)$ and $\theta(k) = \beta(k)\delta(k)$. Now, using Lemma \ref{lem2.1}, $\begin{pmatrix}
	 \alpha & \beta \\ \gamma & \delta
	 \end{pmatrix} \in \mathcal{M}$.

 	This shows that to every $\theta \in End(G)$ we can associate an element $\begin{pmatrix}
	\alpha & \beta\\
	\gamma & \delta
	\end{pmatrix} \in \mathcal{M}$. This defines a map $T: End(G)\longrightarrow \mathcal{M}$ given by $\theta \longmapsto \begin{pmatrix}
	\alpha & \beta\\
	\gamma & \delta
	\end{pmatrix}$. 
	
	Now, we prove that the map $T$ is a monoid homomorphism. Let $\theta, \theta^{\prime}\in End(G)$. Let $\alpha$, $\beta$, $\gamma$ and $\delta$ be the maps associated with $\theta$ and $\alpha^{\prime}$, $\beta^{\prime}$, $\gamma^{\prime}$ and $\delta^{\prime}$ be the maps associated with $\theta^{\prime}$. Then for all $hk\in G$, we have
	\begin{align*}
	\theta^{\prime}\theta(hk) &= \theta^{\prime}(\alpha(h)\gamma(h)\beta(k)\delta(k))\\
	&= \theta^{\prime}(\alpha(h)(\gamma(h)\cdot \beta(k))\gamma(h)^{\beta(k)}\delta(k))\\
	&= \alpha^{\prime}(\alpha(h)(\gamma(h)\cdot \beta(k))\gamma^{\prime}(\alpha(h)(\gamma(h)\cdot \beta(k)))\beta^{\prime}(\gamma(h)^{\beta(k)}\delta(k))\delta^{\prime}(\gamma(h)^{\beta(k)}\delta(k))\\
	&= \alpha^{\prime}(\alpha(h))(\gamma^{\prime}(\alpha(h))\cdot \alpha^{\prime}(\gamma(h)\cdot \beta(k)))\gamma^{\prime}(\alpha(h))^{\alpha^{\prime}(\gamma(h)\cdot \beta(k))}\gamma^{\prime}(\gamma(h)\cdot \beta(k))\\
	&~~~~\beta^{\prime}(\gamma(h)^{\beta(k)})(\delta^{\prime}(\gamma(h)^{\beta(k)})\cdot \beta^{\prime}(\delta(k))){\delta^{\prime}(\gamma(h)^{\beta(k)})}^{ \beta^{\prime}(\delta(k))}\delta^{\prime}\delta(k)\\
	&= \alpha^{\prime}\alpha(h)\gamma^{\prime}\alpha(h)\alpha^{\prime}(\gamma(h)\cdot \beta(k))\gamma^{\prime}(\gamma(h)\cdot \beta(k))\beta^{\prime}(\gamma(h)^{\beta(k)})\delta^{\prime}(\gamma(h)^{\beta(k)})\beta^{\prime}\delta(k)\delta^{\prime}\delta(k)\\
	&= \alpha^{\prime}\alpha(h)\gamma^{\prime}\alpha(h)(\alpha^{\prime}(\gamma(h)\cdot \beta(k))(\gamma^{\prime}(\gamma(h)\cdot \beta(k))\cdot \beta^{\prime}(\gamma(h)^{\beta(k)})))\\
	&~~~~({\gamma^{\prime}(\gamma(h)\cdot \beta(k))}^ {\beta^{\prime}(\gamma(h)^{\beta(k)})})\delta^{\prime}(\gamma(h)^{\beta(k)})\beta^{\prime}\delta(k)\delta^{\prime}\delta(k)\\
	&= \alpha^{\prime}\alpha(h)\gamma^{\prime}\alpha(h)\beta^{\prime}(\gamma(h))(\delta^{\prime}(\gamma(h))\cdot \alpha^{\prime}(\beta(k)))\delta^{\prime}(\gamma(h))^{\alpha^{\prime}(\beta(k))}\gamma^{\prime}\beta(k)\beta^{\prime}\delta(k)\delta^{\prime}\delta(k)\\
&=	\alpha^{\prime}\alpha(h)\gamma^{\prime}\alpha(h)\beta^{\prime}\gamma(h)\delta^{\prime}\gamma(h)\alpha^{\prime}\beta(k)\gamma^{\prime}\beta(k)\beta^{\prime}\delta(k)\delta^{\prime}\delta(k)\\
&= \alpha^{\prime}\alpha(h)(\gamma^{\prime}\alpha(h)\cdot \beta^{\prime}\gamma(h))\gamma^{\prime}\alpha(h)^{\beta^{\prime}\gamma(h)}\delta^{\prime}\gamma(h)\alpha^{\prime}\beta(k)(\gamma^{\prime}\beta(k)\cdot \beta^{\prime}\delta(k))\gamma^{\prime}\beta(k)^{\beta^{\prime}\delta(k)}\delta^{\prime}\delta(k)\\
&= (\alpha^{\prime}\alpha + (\gamma^{\prime}\alpha \cdot \beta^{\prime}\gamma))(h)({\gamma^{\prime}\alpha}^{\beta^{\prime}\gamma} + \delta^{\prime}\gamma)(h)(\alpha^{\prime}\beta + (\gamma^{\prime}\beta \cdot \beta^{\prime}\delta))(k)({\gamma^{\prime}\beta}^{\beta^{\prime}\delta}+ \delta^{\prime}\delta)(k)
	\end{align*}
	Thus, the maps associated with $\theta^{\prime}\theta\in End(G)$ are $a = \alpha^{\prime}\alpha + (\gamma^{\prime}\alpha \cdot \beta^{\prime}\gamma)$, $b= {\gamma^{\prime}\alpha}^{\beta^{\prime}\gamma} + \delta^{\prime}\gamma$, $c = \alpha^{\prime}\beta + (\gamma^{\prime}\beta \cdot \beta^{\prime}\delta)$ and $d = {\gamma^{\prime}\beta}^{\beta^{\prime}\delta}+ \delta^{\prime}\delta$. Therefore, using Lemma \ref{lem2.1}, we get $\begin{pmatrix}
	a& b \\ c & d
	\end{pmatrix}\in \mathcal{M}$ and so
	
	\begin{align*}
	T(\theta^{\prime}\theta) &= \begin{pmatrix}
	\alpha^{\prime}\alpha + (\gamma^{\prime}\alpha \cdot \beta^{\prime}\gamma) & {\gamma^{\prime}\alpha}^{\beta^{\prime}\gamma} +  \delta^{\prime}\gamma\\ \alpha^{\prime}\beta + (\gamma^{\prime}\beta \cdot \beta^{\prime}\delta) & {\gamma^{\prime}\beta}^{\beta^{\prime}\delta}+ \delta^{\prime}\delta
	\end{pmatrix}\\
	&= \begin{pmatrix}
	\alpha^{\prime} & \beta^{\prime}\\ \gamma^{\prime} & \delta^{\prime}
	\end{pmatrix}\begin{pmatrix}
	\alpha & \beta\\ \gamma & \delta
	\end{pmatrix}\\
	&= T(\theta^{\prime})T(\theta).
	\end{align*}
	This shows that the map $T$ is a monoid homomorphism. Now, we prove that the map $T$ is a bijection.
	
	Let $M = \begin{pmatrix}
	\alpha & \beta \\ \gamma & \delta
	\end{pmatrix}\in \mathcal{M}$, where the maps $\alpha$, $\beta$, $\gamma$ and $\delta$ satisfy  the conditions $(i)-(vi)$. Then we define a map $\theta : G \longrightarrow G$ by $\theta(hk) = \alpha(h)\gamma(h)\beta(k)\delta(k)$ for all $h\in H$ and $k\in K$. Let $h, h^{\prime} \in H$ and $k, k^{\prime}\in K$. Then we have
	\begin{align*}
	\theta((h^{\prime}k^{\prime})(hk)) &= \theta(h^{\prime}(k^{\prime}\cdot h){k^{\prime}}^{h}k)\\
	&= \alpha(h^{\prime}(k^{\prime}\cdot h))\gamma(h^{\prime}(k^{\prime}\cdot h))\beta({k^{\prime}}^{h}k)\delta({k^{\prime}}^{h}k)\\
	&= \alpha(h^{\prime})(\gamma(h^{\prime})\cdot \alpha(k^{\prime}\cdot h))\gamma(h^{\prime})^{\alpha(k^{\prime}\cdot h)}\gamma(k^{\prime}\cdot h)\beta({k^{\prime}}^{h})(\delta({k^{\prime}}^{h})\cdot \beta(k))\delta({k^{\prime}}^{h})^{\beta(k)}\delta(k)\\
	&= \alpha(h^{\prime})\gamma(h^{\prime})\alpha(k^{\prime}\cdot h)\gamma(k^{\prime}\cdot h)\beta({k^{\prime}}^{h})\delta({k^{\prime}}^{h})\beta(k)\delta(k)\\
	&= \alpha(h^{\prime})\gamma(h^{\prime})\alpha({k^{\prime}}\cdot h)(\gamma(k^{\prime}\cdot h)\cdot \beta({k^{\prime}}^{h}))\gamma(k^{\prime} \cdot h)^{\beta({k^{\prime}}^{h})}\delta({k^{\prime}}^{h})\beta(k)\delta(k)\\
	&=\alpha(h^{\prime})\gamma(h^{\prime})\beta(k^{\prime})(\delta(k^{\prime})\cdot \alpha(h))\delta(k^{\prime})^{\alpha(h)}\gamma(h)\beta(k)\delta(k)\\
	&= \alpha(h^{\prime})\gamma(h^{\prime})\beta(k^{\prime})\delta(k^{\prime})\alpha(h)\gamma(h)\beta(k)\delta(k)\\
	&=\theta(h^{\prime}k^{\prime})\theta(hk).
		\end{align*}
	Thus $\theta \in End(G)$ such that $T(\theta) = M$. This proves that the map $T$ is a surjection. Now, let $\theta, \theta^{\prime} \in End(G)$ such that $T(\theta) = T(\theta^{\prime}) = \begin{pmatrix}
	\alpha & \beta \\ \gamma & \delta
	\end{pmatrix}\in \mathcal{M}$. Then it is clear that $\theta(hk) = \alpha(h)\gamma(h)\beta(k)\delta(k) = \theta^{\prime}(hk)$ for all $h\in H$ and $k\in K$. Thus $\theta = \theta^{\prime}$. This shows that the map $T$ is injective and so, it is a bijection. Hence the map $T$ is an isomorphism of monoids. 
	\end{proof}

	From here onwards, we will identify the endomorphisms of the semidirect product of groups $H$ and $K$ with the elements of $\mathcal{M}$.
	\section{Determinants}\label{s3}
	In this section, we will define the determinants associated with an endomorphisms of Zappa-Sz\'{e}p product of two groups in the same way as it is defined for the endomorphisms of the semidirect product of two groups (see \cite{detsem}). Let $H$ and $K$ be two groups and $\theta = \begin{pmatrix}
	\alpha & \beta\\ \gamma & \delta
	\end{pmatrix}\in \mathcal{M}$. If $\alpha$ is invertible, then we define $\det_{K}(\theta) = -\gamma\alpha^{-1}\beta + \delta$ to be the $K-$determinant of $\theta$ and if $\delta$ is invertible, then  $\det_{H}(\theta) = \alpha - \beta\delta^{-1} \gamma$ is defined to be the $H-$determinant of $\theta$. In general, $\det_{H}(\theta)$ or $\det_{K}(\theta)$ need neither be a group homomorphism nor a bijection as shown in the following examples. 
	
\begin{example}
	Let $H = \langle x, z \mid z^{3} = 1, xzx^{-1} = z^{-1}\rangle \simeq \mathbb{Z}_{3}\rtimes \mathbb{Z}$ and $K = \langle y \rangle \simeq \mathbb{Z}$. Let the mutual actions $\sigma$ and $\tau$ are defined by 
	\begin{center}
		$\sigma(y,x) = \sigma_{y}(x) = y\cdot x = x$, $\sigma(y, z) = \sigma_{y}(z) = y\cdot z = z^{-1}$, $\tau(y, x) = \tau_{x}(y) = x^{y} = y$ and $\tau(y, z) = \tau_{z}(y) = z^{y} = y$.
	\end{center}
 Thus $G = H\bowtie K$ is a group with the presentation as below
\[G = \langle x, y, z \mid z^{3} = 1, yx = xy, xzx^{-1} = z^{-1}, yz = z^{-1}y\rangle \simeq (\mathbb{Z}_{3}\rtimes \mathbb{Z})\bowtie \mathbb{Z}.\]
Now, define the maps $\alpha: H\longrightarrow H$, $\beta: K \longrightarrow H$, $\gamma: H \longrightarrow K$ and $\delta: K \longrightarrow K$ as below
\begin{align*}
\alpha(x^{i}) = x^{3i} \; \text{and}\;& \alpha(z^{j}) = z^{j}, \delta = I_{K}, \gamma(x) = y \; \text{and}\; \gamma(z) = 1\\
\text{and}\;\;&\beta(y^{l}) = \left\{\begin{array}{ll}
z, \;\; &\text{if $l$ is odd}\\
1, \;\; &\text{if $l$ is even}
\end{array} \right.
\end{align*}
for all integers $i,l$ and $0\le j\le 2$. One can easily check that the maps $\alpha$, $\beta$, $\gamma$ and $\delta$ satisfy the conditions $(i) - (iv)$. Now, for any $i, j$ we have
\begin{align*}
\alpha\left({y^{j}}\cdot x^{i}\right)\left(\gamma\left({y^{j}}\cdot x^{i}\right)\cdot \beta\left({(x^{i})}^{y^{j}}\right)\right) &= \alpha\left(x^{i}\right)\left(\gamma\left(x^{i}\right)\cdot \beta\left(y^{j}\right)\right) \\
&= x^{3i}\left(y^{i}\cdot \beta\left(y^{j}\right)\right) \\
&= \left\{\begin{array}{ll}
x^{3i}\left(y^{i}\cdot z\right), &\; \text{if $j$ is odd}\\
x^{3i}\left(y^{i}\cdot 1\right), &\; \text{if $j$ is even}
\end{array}  \right.\\
&= \left\{\begin{array}{ll}
x^{3i}z^{-1}, &\; \text{if $j$ is odd and $i$ is odd}\\
x^{3i}z, &\; \text{if $j$ is odd and $i$ is even}\\
x^{3i}, &\; \text{if $j$ is even}
\end{array}  \right..
\end{align*}
and \begin{align*}
\beta(y^{j})\left(\delta(y^{j}) \cdot\alpha(x^{i})\right) &= \left\{\begin{array}{ll}
z\left(y^{j}\cdot x^{3i}\right), &\; \text{if $j$ is odd}\\
y^{j}\cdot x^{3i}, &\; \text{if $j$ is even}
\end{array}  \right.\\
&= \left\{\begin{array}{ll}
x^{3i}z^{-1}, &\; \text{if $j$ is odd and $i$ is odd}\\
x^{3i}z, &\; \text{if $j$ is odd and $i$ is even}\\
x^{3i}, &\; \text{if $j$ is even}
\end{array}  \right..
\end{align*}
Also, 
\begin{align*}
\alpha\left({y^{j}}\cdot z^{i}\right)\left(\gamma\left({y^{j}}\cdot z^{i}\right)\cdot \beta\left({(z^{i})}^{y^{j}}\right)\right)
&= \alpha\left({y^{j}}\cdot z^{i}\right)\beta\left(y^{j}\right) \\
&= \left\{\begin{array}{ll}
\alpha(z^{-i})z, &\; \text{if $j$ is odd}\\
\alpha(z^{i}), &\; \text{if $j$ is even}
\end{array}  \right.\\
&= \left\{\begin{array}{ll}
z^{-i+1}, &\; \text{if $j$ is odd}\\
z^{i}, &\; \text{if $j$ is even}
\end{array}  \right..
\end{align*}
and \begin{align*}
\beta(y^{j})\left(\delta(y^{j}) \cdot\alpha(z^{i})\right) &= \beta(y^{j})\left(y^{j} \cdot z^{i}\right)\\
&= \left\{\begin{array}{ll}
z^{-i+1}, &\; \text{if $j$ is odd}\\
z^{i}, &\; \text{if $j$ is even}
\end{array}  \right..
\end{align*}
Further, since $K$ is abelian and $\tau_{h} = I_{K}$ for all $h\in H$, one can easily observe that $(vi)$ holds. Thus the maps $\alpha$, $\beta$, $\gamma$ and $\delta$ satisfy all the conditions $(i)-(vi)$. Therefore, $\begin{pmatrix}
\alpha & \beta\\ \gamma & \delta
\end{pmatrix}\in End(G)$.

Since the map $\delta$ is invertible, the $H$-determinant is defined and is given as
\begin{align*}
(\alpha-\beta{\delta}^{-1}\gamma)(z^{r}) &= z^{r}\\
\text{and}\;\;
(\alpha-\beta{\delta}^{-1}\gamma)(x^{i}) &= \left\{\begin{array}{ll}
x^{3i}z^{-1}, & \; \text{if $i$ is odd}\\
x^{3i}, & \; \text{if $i$ is even}
\end{array}\right..
\end{align*} 
It is evident that the determinant $\alpha-\beta{\delta}^{-1}\gamma$ is neither a group homomorphism nor invertible.
\end{example}
Let $G = H\bowtie K$. In the following theorem, we will prove that for any $\theta\in End(G)$ if either of the determinants $\det_{H}(\theta)$ or $\det_{K}(\theta)$ is defined and invertible, then the map $\theta$ is also invertible. 
	\begin{theorem}\label{t31}
		Let $G = H\bowtie K$ and let $\theta = \begin{pmatrix}
		\alpha & \beta\\ \gamma & \delta
		\end{pmatrix}$ be an endomorphism of $G$. Then the following hold
		\begin{itemize}
			\item[$(i)$] if $\alpha$ is invertible and $\Delta_{K} = \det_{K}(\theta)$ is invertible, then $\theta$ is invertible and 
			\[\theta^{-1} = \begin{pmatrix}
			\alpha^{-1}-\alpha^{-1}\beta{\Delta_{K}}^{-1}(-\gamma\alpha^{-1}) & -\alpha^{-1}\beta{\Delta_{K}}^{-1}\\
			{\Delta_{K}}^{-1}(-\gamma\alpha^{-1}) & {\Delta_{K}}^{-1}
			\end{pmatrix}. \] 
			Moreover, $\det_{H}(\theta^{-1}) = \alpha^{-1}$.
			\item[$(ii)$] if $\delta$ is invertible and $\Delta_{H} = \det_{H}(\theta)$ is invertible, then $\theta$ is invertible and 
			\[\theta^{-1} = \begin{pmatrix}
			{\Delta_{H}}^{-1} & {\Delta_{H}}^{-1}(-\beta\delta^{-1})\\
			-\delta^{-1}\gamma{\Delta_{H}}^{-1}& -\delta^{-1}\gamma{\Delta_{H}}^{-1}(-\beta\delta^{-1}) + \delta^{-1}
			\end{pmatrix}. \] 
			Moreover, $\det_{K}(\theta^{-1}) = \delta^{-1}$. 
		\end{itemize}
	\end{theorem}
	\begin{proof}
		We will only prove the part $(i)$, as the proof of the part $(ii)$ will be on the same lines.
		
		Let $\alpha$ and $\Delta = \Delta_{K}$ be invertible. Then $\Delta\Delta^{-1} = 1 = \Delta^{-1}\Delta$. This gives us
		\begin{equation}\label{e1}
		\gamma\alpha^{-1}\beta\Delta^{-1} + 1 = \delta\Delta^{-1}\; \text{and}\; \Delta^{-1}(-\gamma\alpha^{-1}\beta + \delta) = 1.
		\end{equation}
		Now, let $\alpha^{\prime} = \alpha^{-1}-\alpha^{-1}\beta{\Delta}^{-1}(-\gamma\alpha^{-1})$, $\beta^{\prime} = -\alpha^{-1}\beta{\Delta}^{-1}$, $\gamma^{\prime} = {\Delta}^{-1}(-\gamma\alpha^{-1})$ and $\delta^{\prime} = {\Delta}^{-1}$. Then we have
		
		\begin{align*}
		\alpha\alpha^{\prime} + (\beta\gamma^{\prime}\cdot \gamma\alpha^{\prime}) &=  \alpha(\alpha^{-1}-\alpha^{-1}\beta{\Delta}^{-1}(-\gamma\alpha^{-1})) + (\gamma(\alpha^{-1}-\alpha^{-1}\beta{\Delta}^{-1}(-\gamma\alpha^{-1})
		)\cdot \beta{\Delta}^{-1}(-\gamma\alpha^{-1}))\\
		&= \alpha(\alpha^{-1}-\alpha^{-1}\beta{\Delta}^{-1}(-\gamma\alpha^{-1})) + (\gamma(\alpha^{-1}-\alpha^{-1}\beta{\Delta}^{-1}(-\gamma\alpha^{-1})
		)\cdot \alpha(\alpha^{-1}\beta{\Delta}^{-1}(-\gamma\alpha^{-1})))\\
		&=  \alpha(\alpha^{-1}-\alpha^{-1}\beta{\Delta}^{-1}(-\gamma\alpha^{-1}) + \alpha^{-1}\beta{\Delta}^{-1}(-\gamma\alpha^{-1})) \\
		&= \alpha(\alpha^{-1})\\
		&= 1.\\
		\alpha\beta^{\prime} + (\gamma\beta^{\prime}\cdot \beta\delta^{\prime}) \\
		&= \alpha(-\alpha^{-1}\beta{\Delta}^{-1}) + (\gamma(-\alpha^{-1}\beta{\Delta}^{-1})\cdot \beta{\Delta}^{-1}) \\
		&= \alpha(-\alpha^{-1}\beta{\Delta}^{-1}) + (\gamma(-\alpha^{-1}\beta{\Delta}^{-1})\cdot \alpha(\alpha^{-1}\beta{\Delta}^{-1})) \\
		&= \alpha(-\alpha^{-1}\beta{\Delta}^{-1} + \alpha^{-1}\beta{\Delta}^{-1})\\
		&= 0.\\
		{\gamma\alpha^{\prime}}^{\beta\gamma^{\prime}} + \delta\gamma^{\prime} &= \gamma(\alpha^{-1}-\alpha^{-1}\beta{\Delta}^{-1}(-\gamma\alpha^{-1}))^{\beta\Delta^{-1}(-\gamma\alpha^{-1})} + \delta({\Delta}^{-1}(-\gamma\alpha^{-1}))\\
		&= ((\gamma\alpha^{-1})^{\alpha(-\alpha^{-1}\beta\Delta^{-1}(-\gamma\alpha^{-1}))}+ \gamma(-\alpha^{-1}\beta\Delta^{-1}(-\gamma\alpha^{-1})))^{\beta\Delta^{-1}(-\gamma\alpha^{-1})}\\
		&~~~~ + \delta({\Delta}^{-1}(-\gamma\alpha^{-1}))\\
		&= ((\gamma\alpha^{-1})^{\alpha(-\alpha^{-1}\beta\Delta^{-1}(-\gamma\alpha^{-1}))})^{(\gamma(-\alpha^{-1}\beta\Delta^{-1}(-\gamma\alpha^{-1}))\cdot \beta\Delta^{-1}(-\gamma\alpha^{-1}))}\\
		&~~~~ + (\gamma(-\alpha^{-1}\beta\Delta^{-1}(-\gamma\alpha^{-1})))^{\beta\Delta^{-1}(-\gamma\alpha^{-1})} + \delta({\Delta}^{-1}(-\gamma\alpha^{-1}))\\
		&= ((\gamma\alpha^{-1})^{\alpha(-\alpha^{-1}\beta\Delta^{-1}(-\gamma\alpha^{-1}))(\gamma(-\alpha^{-1}\beta\Delta^{-1}(-\gamma\alpha^{-1}))\cdot \alpha(\alpha^{-1}\beta\Delta^{-1}(-\gamma\alpha^{-1})))}\\
		&~~~~ + (\gamma(-\alpha^{-1}\beta\Delta^{-1}(-\gamma\alpha^{-1})))^{\alpha(\alpha^{-1}\beta\Delta^{-1}(-\gamma\alpha^{-1}))} + \delta({\Delta}^{-1}(-\gamma\alpha^{-1}))\\
		&= \gamma\alpha^{-1} - \gamma\alpha^{-1}\beta\Delta^{-1}(-\gamma\alpha^{-1})) + \delta{\Delta}^{-1}(-\gamma\alpha^{-1})\\
		&= \gamma\alpha^{-1} + (1-\delta\Delta^{-1})(-\gamma\alpha^{-1}) + \delta{\Delta}^{-1}(-\gamma\alpha^{-1})\\
		&= \gamma\alpha^{-1} -\gamma\alpha^{-1} -\delta\Delta^{-1}(-\gamma\alpha^{-1}) + \delta{\Delta}^{-1}(-\gamma\alpha^{-1})\\
		&= 0.\\ 
		{\gamma\beta^{\prime}}^{\beta\delta^{\prime}} + \delta\delta^{\prime} &= \gamma(-\alpha^{-1}\beta{\Delta}^{-1})^{\beta\Delta^{-1}} + \delta{\Delta}^{-1}\\
		&= \gamma(-\alpha^{-1}\beta{\Delta}^{-1})^{\alpha(\alpha^{-1}\beta\Delta^{-1})} + \delta{\Delta}^{-1}\\
		&= -\gamma\alpha^{-1}\beta{\Delta}^{-1}+ \delta{\Delta}^{-1}\\
		&= (-\gamma\alpha^{-1}\beta + \delta){\Delta}^{-1}\\
		&= \Delta{\Delta}^{-1}, \;(\text{using Equation (\ref{e1})})\\
		&= 1.
		\end{align*}
		Thus $\theta^{-1} = \begin{pmatrix}
		\alpha^{-1}-\alpha^{-1}\beta{\Delta_{K}}^{-1}(-\gamma\alpha^{-1}) & -\alpha^{-1}\beta{\Delta_{K}}^{-1}\\
		{\Delta_{K}}^{-1}(-\gamma\alpha^{-1}) & {\Delta_{K}}^{-1}
		\end{pmatrix}$ and $\theta^{-1}\in End(G)$. Moreover, $\delta^{\prime} = \Delta^{-1}$ is invertible. Therefore, $\det_{H}(\theta^{-1})$ is defined and $\det_{H}(\theta^{-1}) = \alpha^{\prime}- \beta^{\prime}{\delta^{\prime}}^{-1}\gamma^{\prime} = \alpha^{-1}-\alpha^{-1}\beta{\Delta}^{-1}(-\gamma\alpha^{-1}) - (-\alpha^{-1}\beta{\Delta}^{-1})\Delta({\Delta}^{-1}(-\gamma\alpha^{-1})) = \alpha^{-1}$. 
	\end{proof}

Let $\mathcal{A}$ be a subset of $\tilde{\mathcal{M}}$ consisting of matrices
\[\begin{pmatrix}
\alpha & \beta\\ \gamma & \delta
\end{pmatrix}\]
such that for all $h, h^{\prime}\in H$ and $k, k^{\prime}\in K$, the following hold:
	\begin{itemize}
	\item[$(A_{1})$] $\alpha(hh^{\prime}) = \alpha(h)(\gamma(h)\cdot \alpha(h^{\prime}))$,
	\item[$(A_{2})$] $\gamma(hh^{\prime}) = \gamma(h)^{\alpha(h^{\prime})}\gamma(h^{\prime})$,
	\item[$(A_{3})$] $\beta(kk^{\prime}) = \beta(k)(\delta(k)\cdot \beta(k^{\prime}))$,
	\item[$(A_{4})$] $\delta(kk^{\prime}) = \delta(k)^{\beta(k^{\prime})}\delta(k^{\prime})$,
	\item[$(A_{5})$] $\beta(k)(\delta(k)\cdot \alpha(h)) = \alpha(k\cdot h)(\gamma(k\cdot h)\cdot \beta(k^{h}))$,
	\item[$(A_{6})$] $\delta(k)^{\alpha(h)}\gamma(h) = \gamma(k\cdot h)^{\beta(k^{h})}\delta(k^{h})$, 
	\item[$(A_7)$] For any $h^{\prime}k^{\prime}\in G$, there exists a unique $h\in H$ and $k\in K$ such that $h^{\prime} = \alpha(h)(\gamma(h)\cdot \beta(k))$ and $k^{\prime} = \gamma(h)^{\beta(k)}\delta(k)$
	\end{itemize}
	for all $h,h^{\prime}\in H$ and $k, k^{\prime}\in K$.
	
	One can easily observe that $\mathcal{A}$ is a group in the monoid $\mathcal{M}$. Note that a one to one correspondence between $Aut(G)$ and $\mathcal{A}$ is already discussed in \cite{zap}. 

Also, for any $\begin{pmatrix}
\alpha & \beta\\ \gamma & \delta
\end{pmatrix}\in \mathcal{A}$, the maps $\alpha$ and $\delta$ need not be a bijection as can be seen in the following example.
\begin{example}\label{ex1}
	Let $H = \langle x, z \mid z^{3} = 1, xz = zx\rangle \simeq \mathbb{Z}\times \mathbb{Z}_{3}$ and $K = \langle y \rangle \simeq \mathbb{Z}$. Let the mutual actions $\sigma$ and $\tau$ be defined by 
	\begin{equation*}
	\tau_{h} = I_{K}, \sigma(y, x) = \sigma_{y}(x) = x\; \text{and}\; \sigma(y, z) = \sigma_{y}(z) = z^{-1}.
	\end{equation*}
	Thus $G = H \bowtie K$ with the presentation given as
	\[G =  \langle x,y, z \mid z^{3} = 1, xz = zx, xy = yx, yzy^{-1} = z^{-1}\rangle \simeq (\mathbb{Z}\times \mathbb{Z}_{3})\bowtie \mathbb{Z}.\] 
Now, let the maps $\alpha: H \longrightarrow H$, $\beta: K \longrightarrow H$, $\gamma: H \longrightarrow K$ and $\delta : K \longrightarrow K$ be defined by 
	\begin{align*}
	\alpha(x) = x^{3} \;\; \text{and}\;\; \alpha(z) = z, &\;\;\;\beta(y) = x^{4}\\
	\gamma(x) = y^{2}\;\; \text{and}\;\; \gamma(z) = 1, &\;\;\;\delta(y) = y^{3}.
	\end{align*}
	Evidently, the maps $\alpha$, $\beta$, $\gamma$ and $\delta$ satisfy all the conditions $(A_{1}) - (A_{7})$. Thus $\begin{pmatrix}
	\alpha & \beta\\ \gamma & \delta
	\end{pmatrix} \in \mathcal{A}$. However, the maps $\alpha$ and $\delta$ are not bijections.
\end{example}
Now, let us consider some important subsets of $Aut(G)$ given as below 
\begin{align*}
P &= \left\{\alpha\in Aut(H) \mid \alpha(k\cdot h) = k\cdot \alpha(h) \; \text{and}\; k^{\alpha(h)} = k^{h}, \forall\; h\in H, k\in K\right\},\\
Q &= \left\{\beta\in CHom(K,H) \mid k = k^{\beta(k^{\prime})}, \beta(k)(k\cdot h) = (k\cdot h)\beta(k^{h}), \forall\; h\in H, k\in K\right\},\\
R &= \left\{\gamma \in CHom(H,K) \mid h^{\prime} = \gamma(h)\cdot h^{\prime}, \gamma(k\cdot h)k^{h} = k^{h}\gamma(h), \forall\; h\in H, k\in K\right\},\\
S &= \left\{\delta\in Aut(K) \mid \delta(k)\cdot h = k\cdot h, \delta(k)^{h} = \delta(k^{h}), \forall\; h\in H, k\in K\right\}
\end{align*}
One can easily note that $P$, $Q$, $R$ and $S$ are the subgroups of the group  $Aut(G)$. In addition, it is easy to check that $P$ and $S$ normalizes both $Q$ and $R$. On the other hand, the corresponding subsets of the group $\mathcal{A}$ are 
	\begin{align*}
	A &= \left\{\begin{pmatrix}
	\alpha & 0 \\ 0 & 1
	\end{pmatrix} \mid \alpha \in P \right\},\\
	B &= \left\{\begin{pmatrix}
	1 & \beta \\ 0 & 1
	\end{pmatrix} \mid \beta \in Q \right\},\\
	C &= \left\{\begin{pmatrix}
	1 & 0 \\ \gamma & 1
	\end{pmatrix} \mid \gamma \in R \right\},\\
	D &= \left\{\begin{pmatrix}
	1 & 0 \\ 0 & \delta
	\end{pmatrix} \mid \delta \in S \right\}.
	\end{align*}
	One can easily note that $A$, $B$, $C$ and $D$ are the subgroups of the group $\mathcal{A}$. Additionally, it is easy to check that $A$ and $D$ normalize both $B$ and $C$. 

\begin{lemma}\label{lem1}
	Let $\begin{pmatrix}
	1 & \beta \\ \gamma & 1
	\end{pmatrix}\in \mathcal{A}$. Then for all $h\in H$ and $k,k^{\prime}\in K$, the following hold:
	\begin{itemize}
		\item[$(i)$] $[Im(\beta), Im(\beta\gamma)] = 1$,
		\item[$(ii)$] $\beta\gamma(k\cdot h)  = k^{h}\cdot \beta\gamma(h)$,
		\item[$(iii)$] $(k\cdot \beta(k^{\prime}))^{-1} = k\cdot (\beta(k^{\prime}))^{-1}$. 
	\end{itemize}
\end{lemma}
\begin{proof}
	Let $h\in H$ and $k\in K$. Then
	\begin{itemize}
		\item[$(i)$] 
		\begin{align*}
		\beta\gamma(h) \beta(k) &= \beta(\gamma(h))(\gamma(h)\cdot \beta(k))\; \text{(using $(A_{1})$)}\\
		&= (\gamma(h)\cdot \beta(k))\beta\left(\gamma(h)^{\beta(k)}\right)\; \text{(using $(A_{5})$)}\\
		&= \beta(k)\beta\gamma(h)\; \text{(using $(A_{1})$ and $(A_{4})$)}.
		\end{align*}
		Hence $[Im(\beta), Im(\beta\gamma)] = 1$.
		\item[$(ii)$] 
		\begin{align*}
		k^{h}\cdot \beta\gamma(h) &= \left((k^{h}\cdot \beta\gamma(h))\beta((k^{h})^{\beta\gamma(h)})\right)\beta((k^{h})^{\beta\gamma(h)})^{-1}\\
		&= \beta(k^{h})(k^{h}\cdot \beta\gamma(h))\beta(k^{h})^{-1}\\
		&= \beta(k^{h}\gamma(h))\beta(k^{h})^{-1}\\
		&= \beta(\gamma(k\cdot h)k^{h})\beta(k^{h})^{-1}\; \text{(using $(A_{6})$)}\\
		&= \beta(\gamma(k\cdot h))(\gamma(k\cdot h)\cdot \beta(k^{h})\beta(k^{h})^{-1}\\
		&= \beta\gamma(k\cdot h).
		\end{align*}
		\item[$(iii)$] 
		\begin{align*}
			(k\cdot \beta(k^{\prime}))(k\cdot \beta(k^{\prime})^{-1}) &= 	(k\cdot \beta(k^{\prime}))(k^{\beta(k^{\prime})}\cdot \beta(k^{\prime})^{-1})\\
			&= k\cdot (\beta(k^{\prime})\beta(k^{\prime})^{-1})\\
			&= k\cdot 1\\
			&= 1 
		\end{align*}
	 Hence $(k\cdot \beta(k^{\prime})^{-1}) = k\cdot \beta(k^{\prime})^{-1}$
	\end{itemize}

\end{proof}

	\begin{lemma}\label{lem2}
		Let $\begin{pmatrix}
		1 & \beta\\ \gamma & 1
		\end{pmatrix}\in \mathcal{A}$. Then the following hold
		\begin{itemize}
			\item[$(i)$] $\begin{pmatrix}
			1-\beta\gamma & 0\\ 0 & 1
			\end{pmatrix}\in A$, 
			\item[$(ii)$] $\begin{pmatrix}
			1 & (1-\beta\gamma)^{-1}\beta \\ 0 & 1
			\end{pmatrix}\in B$.
			\item[$(iii)$] $\begin{pmatrix}
			1 & 0\\ \gamma & 1
			\end{pmatrix}\in C$.
		\end{itemize}
	\end{lemma}
	\begin{proof}
			Let $\begin{pmatrix}
			1 & \beta\\ \gamma & 1
			\end{pmatrix}\in \mathcal{A}$. Then using conditions $(A_{1})-(A_{6})$, we have $k = k^{\beta(k^{\prime})}, \beta(k)(k\cdot h) = (k\cdot h)\beta(k^{h}),h^{\prime} = \gamma(h)\cdot h^{\prime}, \gamma(k\cdot h)k^{h} = k^{h}\gamma(h)$ for all $h, h^{\prime}\in H$ and $k, k^{\prime}\in K$. We will use these equalities throughout the proof.
			
			\begin{itemize}
				\item[$(i)$]   	First, we prove that $1-\beta\gamma$ is a group homomorphism. Let $h, h^{\prime}\in H$. Then
			\begin{align*}
			(1-\beta\gamma)(hh^{\prime}) &= hh^{\prime}(\beta\gamma(hh^{\prime}))^{-1}\\
			&= hh^{\prime}(\beta(\gamma(h)^{h^{\prime}}\gamma(h^{\prime})))^{-1}\\
			&= hh^{\prime}\left(\beta(\gamma(h)^{h^{\prime}})(\gamma(h)^{h^{\prime}}\cdot \beta\gamma(h^{\prime}))\right)^{-1}\\
			&= hh^{\prime}\left((\gamma(h)^{h^{\prime}}\cdot \beta\gamma(h^{\prime}))\beta\left((\gamma(h)^{h^{\prime}})^{\beta\gamma(h^{\prime})}\right)\right)^{-1}\\
			&= hh^{\prime}\left((\gamma(h)\cdot h^{\prime})^{-1}\left((\gamma(h)\cdot h^{\prime})(\gamma(h)^{h^{\prime}}\cdot \beta\gamma(h^{\prime}))\right)\beta(\gamma(h)^{h^{\prime}})\right)^{-1}\\
			&= hh^{\prime}\left({h^{\prime}}^{-1}(\gamma(h)\cdot h^{\prime}\beta\gamma(h^{\prime}))\beta(\gamma(h)^{h^{\prime}})\right)^{-1}\\
			&= hh^{\prime}\left({h^{\prime}}^{-1}(h^{\prime}\beta\gamma(h^{\prime}))\beta(\gamma(h)^{h^{\prime}})\right)^{-1}\\
			&= hh^{\prime}\left(\beta\gamma(h^{\prime})(\gamma(h)\cdot h^{\prime})^{-1}\left((\gamma(h)\cdot h^{\prime})\beta(\gamma(h)^{h^{\prime}})\right)\right)^{-1}\\
			&= hh^{\prime}\left(\beta\gamma(h^{\prime}){h^{\prime}}^{-1}\beta\gamma(h)(\gamma(h)\cdot h^{\prime})\right)^{-1}\\
			&= hh^{\prime}\left(\beta\gamma(h^{\prime}){h^{\prime}}^{-1}\beta\gamma(h)h^{\prime}\right)^{-1}\\
			&= h(\beta\gamma(h))^{-1}h^{\prime}(\beta\gamma(h^{\prime}))^{-1}\\
			&= (1-\beta\gamma)(h)(1-\beta\gamma)(h^{\prime}).
			\end{align*}
			Thus $1-\beta\gamma$ is a group homomorphism. Now, let $h^{\prime}\in H$ be any element. Then using $(A_{7})$, there exists a unique $h\in H$ and $k\in K$ such that $h^{\prime} = h(\gamma(h)\cdot \beta(k))$ and $1 = \gamma(h)^{\beta(k)}k$. This implies that $\gamma(h)k = 1$. Therefore, $h^{\prime} = h(\gamma(h)\cdot \beta(\gamma(h)^{-1})) = h\beta(\gamma(h))^{-1} = (1-\beta\gamma)(h)$. Thus $1-\beta\gamma$ is a surjection. Also, the uniqueness of elements $h\in H$ and $k\in K$ in $(A_{7})$ gives that $1-\beta\gamma$ is an injective map. Thus $1-\beta\gamma \in Aut(H)$.

     Now, for all $h\in H$ and $k\in K$, we have
     \begin{align*}
     (1-\beta\gamma)(k\cdot h) &= (k\cdot h)\left(\beta\gamma(k\cdot h)\right)^{-1}\\
     &= (k\cdot h)\left(k^{h}\cdot \beta\gamma(h)\right)^{-1}\; (\text{using Lemma \ref{lem1}$(ii)$})\\
     &= (k\cdot h)\left(k^{h}\cdot \beta\gamma(h)^{-1}\right)\; (\text{using Lemma \ref{lem1}$(iii)$})\\
     &= k\cdot(h\beta\gamma(h)^{-1})\\
     &= k\cdot (1-\beta\gamma)(h).
     \end{align*}
Also, we have $k^{(1-\beta\gamma)(h)} = k^{h\beta\gamma(h)^{-1}} = (k^{h})^{\beta(\gamma(h)^{-1})} = k^{h}$. Thus $(1-\beta\gamma)\in P$ and so $\begin{pmatrix}
1-\beta\gamma & 0 \\ 0 & 1
\end{pmatrix}\in A$.
			\item[$(ii)$] Using part $(i)$, we get $\begin{pmatrix}
			(1-\beta\gamma)^{-1} & 0\\ 0 & 1
			\end{pmatrix}\in A$. Let $k, k^{\prime}\in K$. Then
			\begin{align*}
			(1-\beta\gamma)^{-1}\beta(kk^{\prime}) &= (1-\beta\gamma)^{-1}\left(\beta(k)(k\cdot \beta(k^{\prime}))\right)\\
			&= (1-\beta\gamma)^{-1}(\beta(k))(1-\beta\gamma)^{-1}\left(k\cdot \beta(k^{\prime})\right)\\
			&= (1-\beta\gamma)^{-1}\beta(k)(k\cdot (1-\beta\gamma)^{-1}\beta(k^{\prime})).
			\end{align*}
			Therefore, $(1-\beta\gamma)^{-1}\beta\in CHom(K, H)$. Also, $k^{(1-\beta\gamma)^{-1}\beta(k^{\prime})} = k^{\beta(k^{\prime})} = k$. At the last, we have 			
			\begin{align*}
			(1-\beta\gamma)^{-1}\beta(k)(k\cdot h) &= (1-\beta\gamma)^{-1}(\beta(k)(1-\beta\gamma)(k\cdot h))\\
			&= (1-\beta\gamma)^{-1}(\beta(k)(k\cdot (1-\beta\gamma)(h)))\\
			 &= (1-\beta\gamma)^{-1}((k\cdot (1-\beta\gamma)(h))\beta(k^{(1-\beta\gamma)(h)}))\\
			 &= (1-\beta\gamma)^{-1}((k\cdot (1-\beta\gamma)(h))\beta(k^{h}))\\
			 &= (k\cdot h)(1-\beta\gamma)^{-1}\beta(k^{h}).
			\end{align*}
		 Hence $(1-\beta\gamma)^{-1}\beta \in Q$ and so $\begin{pmatrix}
			1 & (1-\beta\gamma)^{-1}\beta \\ 0 & 1
			\end{pmatrix}\in B$.
			\item[$(iii)$] Using $(A_{1}), (A_{2})$ and $(A_{6})$, it is easy to observe that $\begin{pmatrix}
			1 &0\\ \gamma & 1
			\end{pmatrix}\in C$.
		\end{itemize}
	\end{proof}
	\begin{theorem}\label{th32}
		Let $G= H\bowtie K$ and let $\mathcal{A}$ be defined as above. Then $\mathcal{A} = ABCD$.
	\end{theorem}
	\begin{proof}
Let $\begin{pmatrix}
1 & \beta\\ \gamma & 1
\end{pmatrix} \in \mathcal{A}$. Then using Lemma \ref{lem2}, we get
		\begin{equation}\label{e3}
		\begin{pmatrix}
		1 & \beta\\ \gamma & 1
		\end{pmatrix} = \begin{pmatrix}
		1-\beta\gamma & 0\\ 0 & 1
		\end{pmatrix}\begin{pmatrix}
		1 & (1-\beta\gamma)^{-1}\beta\\ 0 & 1
		\end{pmatrix}\begin{pmatrix}
		1 & 0\\ \gamma & 1
		\end{pmatrix}\in ABC.
		\end{equation}
		Now, let $\begin{pmatrix}
		\alpha & \beta \\ \gamma & \delta
		\end{pmatrix} \in \mathcal{A}$, where the maps satisfy the conditions $(A_{1})-(A_{7})$. Then 
		\begin{equation}\label{e4}
		\begin{pmatrix}
		\alpha & \beta\\ \gamma & \delta
		\end{pmatrix} = \begin{pmatrix}
		\alpha & 0\\ 0 & 1
		\end{pmatrix}\begin{pmatrix}
		1 & \alpha^{-1}\beta\delta^{-1}\\ \gamma & 1
		\end{pmatrix}\begin{pmatrix}
		1 & 0\\ 0 & \delta
		\end{pmatrix}.
		\end{equation}
		Note that $\begin{pmatrix}
		\alpha^{-1} & 0\\ 0 & 1
		\end{pmatrix}\in A$ and $\begin{pmatrix}
		1 & 0\\ 0 & \delta^{-1}
		\end{pmatrix}\in D$. Then for all $k,k^{\prime}\in K$, we have
		\begin{align*}
		\alpha^{-1}\beta\delta^{-1}(kk^{\prime}) &= \alpha^{-1}\beta\left(\delta^{-1}(k)\delta^{-1}(k^{\prime})\right)\\
		&= \alpha^{-1}\left(\beta(\delta^{-1}(k))\left(\delta(\delta^{-1}(k))\cdot \beta(\delta^{-1}(k^{\prime})) \right) \right)\\
		&= \alpha^{-1}\left(\beta\delta^{-1}(k)\left(k\cdot \beta\delta^{-1}(k^{\prime}) \right) \right)\\
		&= \alpha^{-1}\beta\delta^{-1}(k)\left(k\cdot \alpha^{-1}\beta\delta^{-1}(k^{\prime})\right)
		\end{align*}
		Thus $\alpha^{-1}\beta\delta^{-1} \in CHom(K, H)$. Now,
		\begin{align*}
		k^{\alpha^{-1}\beta\delta^{-1}(k^{\prime})} &= k^{\alpha^{-1}(\beta\delta^{-1}(k^{\prime}))}\\
		&= k^{\beta(\delta^{-1}(k^{\prime}))}\\
		&= (\delta(\delta^{-1}(k)))^{\beta(\delta^{-1}(k^{\prime}))}\delta(\delta^{-1}(k^{\prime}))(k^{\prime})^{-1}\\
		&= \delta(\delta^{-1}(k)\delta^{-1}(k^{\prime}))(k^{\prime})^{-1}\\
		&= \delta(\delta^{-1}(kk^{\prime}))(k^{\prime})^{-1}\\ 
		&= k.
		\end{align*} 
		 Also, we have
		\begin{align*}
		\alpha^{-1}\beta\delta^{-1}(k)(k\cdot h) &= \alpha^{-1}\beta\delta^{-1}(k)(k\cdot \alpha^{-1}(\alpha(h)))\\
		&=  \alpha^{-1}\beta\delta^{-1}(k)\alpha^{-1}(k\cdot \alpha(h))\\
		&= 	\alpha^{-1}\left(\beta(\delta^{-1}(k))(\delta(\delta^{-1}(k))\cdot \alpha(h)) \right)\\
		&= \alpha^{-1}\left(\alpha(\delta^{-1}(k)\cdot h)(\gamma(\delta^{-1}(k)\cdot h)\cdot \beta({\delta(k)}^{h})) \right)\\
		&= \alpha^{-1}\left(\alpha(\delta^{-1}(k)\cdot h)(\gamma(\delta^{-1}(k)\cdot h)\cdot \alpha(\alpha^{-1}\beta({\delta^{-1}(k)}^{h}))) \right)\\
		&= \alpha^{-1}\left(\alpha(\delta^{-1}(k)\cdot h)\left(\gamma(\delta^{-1}(k)\cdot h)\cdot \alpha(\alpha^{-1}\beta\delta^{-1}(k^{h})\right) \right)\\
		&= \alpha^{-1}\left(\alpha\left((\delta^{-1}(k)\cdot h)\alpha^{-1}\beta\delta^{-1}(k^{h})\right) \right)\\
		&= (\delta^{-1}(k)\cdot h)\alpha^{-1}\beta\delta^{-1}(k^{h})\\
		&= (k\cdot h)\alpha^{-1}\beta\delta^{-1}(k^{h}).
		\end{align*}
		 Therefore, $\alpha^{-1}\beta\delta^{-1}\in Q$ and so $\begin{pmatrix}
		1 & \alpha^{-1}\beta\delta^{-1}\\ 0 & 1
		\end{pmatrix}\in B$. By Equations (\ref{e3}) and (\ref{e4}), we get
		\[\begin{pmatrix}
		\alpha & \beta\\ \gamma & \delta
		\end{pmatrix} = \begin{pmatrix}
		\alpha & 0\\ 0 & 1
		\end{pmatrix}\begin{pmatrix}
		1 & \alpha^{-1}\beta\delta^{-1}\\ \gamma & 1
		\end{pmatrix}\begin{pmatrix}
		1 & 0\\ 0 & \delta
		\end{pmatrix}\in A(ABC)D \subseteq ABCD. \]
		Thus $\mathcal{A} \subseteq ABCD$. Also, it is evident that $ABCD \subseteq \mathcal{A}$. Hence $\mathcal{A} = ABCD$.

	\end{proof}
	
		\begin{theorem}\label{t32}
		Let $G = H\bowtie K$ and $\theta = \begin{pmatrix}
		\alpha & \beta\\ \gamma & \delta
		\end{pmatrix}\in End(G)$ such that $\alpha$ is invertible. Then $\theta$ is invertible if and only if $\det_{K}(\theta)$ is invertible.  
	\end{theorem}
	\begin{proof}
		Let $\alpha$ be invertible. First assume that $\theta$ is invertible. Now, let $k, k^{\prime}\in K$, such that $\det_{K}(\theta)(k) = \det_{K}(\theta)(k^{\prime})$. Then
		\begin{equation}\label{e7}
		(\gamma\alpha^{-1}\beta(k))^{-1}\delta(k) = (\gamma\alpha^{-1}\beta(k^{\prime}))^{-1}\delta(k^{\prime})
		\end{equation}  
		Now,
		\begin{align*}
		\begin{pmatrix}
		\alpha & \beta \\ \gamma & \delta
		\end{pmatrix}\begin{pmatrix}
		(\alpha^{-1}\beta(k))^{-1}\\ k
		\end{pmatrix} &= \begin{pmatrix}
		\alpha((\alpha^{-1}\beta(k))^{-1})({\gamma((\alpha^{-1}\beta(k))^{-1})}\cdot \beta(k))\\ ({\gamma((\alpha^{-1}\beta(k))^{-1})})^{\beta(k)}\delta(k)
		\end{pmatrix}\\
		&= \begin{pmatrix}
		\alpha((\alpha^{-1}\beta(k))^{-1})({\gamma((\alpha^{-1}\beta(k))^{-1})}\cdot \alpha(\alpha^{-1}\beta(k)))\\ ({\gamma((\alpha^{-1}\beta(k))^{-1})})^{\alpha(\alpha^{-1}\beta(k))}\delta(k)
		\end{pmatrix}\\
		&= \begin{pmatrix}
		\alpha((\alpha^{-1}\beta(k))^{-1}\alpha^{-1}\beta(k))\\ \gamma(\alpha^{-1}\beta(k))^{-1}\delta(k)
		\end{pmatrix}\\
		&= \begin{pmatrix}
		1\\ (\gamma\alpha^{-1}\beta(k))^{-1}\delta(k)
		\end{pmatrix}.
		\end{align*}
		Similarly, we get 
		\begin{equation*}
		\begin{pmatrix}
		\alpha & \beta \\ \gamma & \delta
		\end{pmatrix}\begin{pmatrix}
		(\alpha^{-1}\beta(k^{\prime}))^{-1}\\ k^{\prime}
		\end{pmatrix}  = \begin{pmatrix}
		1\\ (\gamma\alpha^{-1}\beta(k^{\prime}))^{-1}\delta(k^{\prime})
		\end{pmatrix}.
		\end{equation*}
		Using Equation (\ref{e7}), we get
		\[\begin{pmatrix}
		\alpha & \beta \\ \gamma & \delta
		\end{pmatrix}\begin{pmatrix}
		(\alpha^{-1}\beta(k))^{-1}\\ k
		\end{pmatrix} = \begin{pmatrix}
		\alpha & \beta \\ \gamma & \delta
		\end{pmatrix}\begin{pmatrix}
		(\alpha^{-1}\beta(k^{\prime}))^{-1}\\ k^{\prime}
		\end{pmatrix}. \]
		Since $\theta$ is invertible, $k = k^{\prime}$. This shows that $\det_{K}(\theta)$ is injective.
		
		Now, let $k^{\prime\prime}\in K$ be any element. Since $\theta$ is invertible, we have $hk\in G$ such that $\theta(hk) = k^{\prime\prime}$. Then $\alpha(h)(\gamma(h)\cdot \beta(k))\gamma(h)^{\beta(k)}\delta(k) = k^{\prime\prime}$. Using the uniqueness of representation, we have $\alpha(h)(\gamma(h)\cdot \beta(k)) = 1$ and $\gamma(h)^{\beta(k)}\delta(k) = k^{\prime\prime}$. This implies that $\alpha(h\alpha^{-1}\beta(k)) = 1$. Since $\alpha$ is invertible, $h\alpha^{-1}\beta(k) = 1$ and so, $h = (\alpha^{-1}\beta(k))^{-1}$. Using $\gamma(h)^{\beta(k)}\delta(k) = k^{\prime\prime}$, we get $k^{\prime\prime} = \gamma((\alpha^{-1}\beta(k))^{-1})^{\alpha(\alpha^{-1}\beta(k))}\delta(k) = \gamma(\alpha^{-1}\beta(k))^{-1}\delta(k) = (-\gamma\alpha^{-1}\beta + \delta)(k)$. Thus $\det_{K}(\theta)$ is surjective. Hence $\det_{K}(\theta)$ is invertible. The converse holds using Theorem \ref{t31}.
	\end{proof}
	\begin{theorem}\label{t33}
		Let $G = H\bowtie K$ and $\theta = \begin{pmatrix}
		\alpha & \beta\\ \gamma & \delta
		\end{pmatrix}\in End(G)$ such that $\delta$ is invertible. Then $\theta$ is invertible if and only if $\det_{H}(\theta)$ is invertible.  
	\end{theorem}
	\begin{proof}
		The proof is on the lines of the proof of Theorem \ref{t32}.
	\end{proof}
%
%
	
	\begin{theorem}\label{t34}
		Let $G = H\bowtie K$ and $\theta = \begin{pmatrix}
		\alpha & \beta \\ \gamma & \delta
		\end{pmatrix}\in \mathcal{M}$ such that both $\alpha$ and $\delta$ are invertible. Then $\Delta_{H} = \det_{H}(\theta)$ is invertible if and only if $\Delta_{K}= \det_{K}(\theta)$ is invertible. Moreover, 
		\begin{align}
		{\Delta_{H}}^{-1} &= \alpha^{-1} - \alpha^{-1}\beta{\Delta_{K}}^{-1}(-\gamma\alpha^{-1}) \label{e5}\\
		\text{and}\hspace{.5cm} {\Delta_{K}}^{-1} &= -\delta^{-1}\gamma{\Delta_{H}}^{-1}(-\beta\delta^{-1})+ \delta^{-1}. \label{e6}
		\end{align}
	\end{theorem}
	\begin{proof}
		Let both $\alpha$ and $\delta$ are invertible. First, let $\Delta_{K}$ be invertible. Then we prove that $\Delta_{H}$ is invertible and Equation (\ref{e5}) holds. Let $h\in H$. Then $\gamma\alpha^{-1}(h)\in K$. Since $\Delta_{K}$ is invertible, there exists $k\in K$ such that $\Delta_{K}(k) = (\gamma\alpha^{-1}(h))^{-1}$. This implies that $\gamma\alpha^{-1}\beta(k) = \delta(k)\gamma\alpha^{-1}(h)$. Now,
		\begin{align*}
		\Delta_{H}(\alpha^{-1}(h)(\alpha^{-1}\beta(k))^{-1}) &= (\alpha-\beta\delta^{-1}\gamma)(\alpha^{-1}(h)(\alpha^{-1}\beta(k))^{-1})\\
		&= \alpha(\alpha^{-1}(h)(\alpha^{-1}\beta(k))^{-1})(\beta\delta^{-1}\gamma(\alpha^{-1}(h)(\alpha^{-1}\beta(k))^{-1}))^{-1}\\
		&= \alpha(\alpha^{-1}(h))(\gamma\alpha^{-1}(h)\cdot \alpha((\alpha^{-1}\beta(k))^{-1}))(\beta\delta^{-1}(\gamma\alpha^{-1}(h)^{\alpha((\alpha^{-1}\beta(k))^{-1})}\\
		&~~~~\gamma((\alpha^{-1}\beta(k))^{-1})))^{-1}\\
		&= h(\delta(k)^{-1}\gamma\alpha^{-1}\beta(k)\cdot \alpha((\alpha^{-1}\beta(k))^{-1}))(\beta\delta^{-1}((\delta(k)^{-1}\gamma\alpha^{-1}\beta(k))^{\alpha((\alpha^{-1}\beta(k))^{-1})}\\
		&~~~~\gamma((\alpha^{-1}\beta(k))^{-1})))^{-1}\\
		&= h(\delta(k)^{-1}\cdot (\gamma\alpha^{-1}\beta(k)\cdot \alpha((\alpha^{-1}\beta(k))^{-1})))(\beta\delta^{-1}((\delta(k)^{-1})^{\gamma\alpha^{-1}\beta(k)\cdot \alpha((\alpha^{-1}\beta(k))^{-1})}\\
		&~~~~\gamma\alpha^{-1}\beta(k)^{ \alpha((\alpha^{-1}\beta(k))^{-1})}\gamma((\alpha^{-1}\beta(k))^{-1})))^{-1}\\
		&= h(\delta(k)^{-1}\cdot (\alpha(\alpha^{-1}\beta(k)))^{-1})(\beta\delta^{-1}((\delta(k)^{-1})^{(\alpha(\alpha^{-1}\beta(k)))^{-1}}\\
		&~~~~\gamma(\alpha^{-1}\beta(k)(\alpha^{-1}\beta(k))^{-1})))^{-1}\\
		&= h(\delta(k)^{-1}\cdot \beta(k)^{-1})(\beta\delta^{-1}((\delta(k^{-1})^{\beta(k)})^{\beta(k)^{-1}}))^{-1}\\
		&= h(\delta(k)^{-1}\cdot(\delta(k)\cdot \beta(k^{-1})))\beta(k^{-1})^{-1}\\
		&= h \beta(k^{-1})\beta(k^{-1})^{-1}\\
		&= h
		\end{align*}
		
		Thus for $h\in H$, there exists $\alpha^{-1}(h)(\alpha^{-1}\beta{\Delta_{K}}^{-1}(\gamma\alpha^{-1}(h)^{-1}))^{-1}) \in H$ such that $$\Delta_{H}(\alpha^{-1}(h)(\alpha^{-1} \beta{\Delta_{K}}^{-1}(\gamma \alpha^{-1}(h)^{-1}))^{-1}) = h.$$ Therefore, ${\Delta_{H}}^{-1}(h) = (\alpha^{-1} - \alpha^{-1}\beta{\Delta_{K}}^{-1}(-\gamma\alpha^{-1}))(h)$. Hence ${\Delta_{H}}^{-1} = \alpha^{-1}-\alpha^{-1}\beta{\Delta_{K}}^{-1}(-\gamma\alpha^{-1})$. 	
		
		Using the similar arguments, one can easily prove Equation (\ref{e6}).
		
%
	\end{proof}
	Using Theorem \ref{t34}, we can define the determinant of an endomorphism of the group $G= H\bowtie K$ as follows: Let $G = H\bowtie K$ and $\theta \in \mathcal{A}$ such that both $\alpha$ and $\delta$ are invertible. Then we define the determinant of $\theta$ to be the map $\alpha - \beta\delta^{-1}\gamma$ and denote it by $\det(\theta)$.

	\begin{theorem}\label{t36}
		Let $G = H\bowtie K$ and $\theta= \begin{pmatrix}
		\alpha & \beta\\ \gamma & \delta
		\end{pmatrix}\in \mathcal{A}$ such that both $\alpha$ and $\delta$ are invertible. If $\Delta = \det(\theta)$ is invertible, then 
		\[\theta^{-1} = \begin{pmatrix}
		{\Delta}^{-1} & {\Delta}^{-1}(-\beta\delta^{-1})\\
		-\delta^{-1}\gamma{\Delta}^{-1}& -\delta^{-1}\gamma{\Delta}^{-1}(-\beta\delta^{-1}) + \delta^{-1}
		\end{pmatrix}.\]
		Moreover, $\theta^{-1}\in \mathcal{A}$ and $\det(\theta^{-1}) = \alpha^{-1}$.
	\end{theorem}
	\begin{proof}
		Let $\Delta_{H}$ be invertible. Then the formula for $\theta^{-1}$ can be obtained directly from Theorem $\ref{t31}$. Using Theorem \ref{t34}, we have $\det_{K}(\theta)$ is invertible. Also, using the two formulas for $\theta^{-1}$ in Theorem \ref{t31}, we get ${\det_{K}}^{-1}(\theta) = \delta^{-1}\gamma{{\Delta}_{H}}^{-1}(-\beta\delta^{-1})+ \delta^{-1}$. Thus $\theta^{-1}\in \mathcal{A}$. Further, using $\theta = (\theta^{-1})^{-1}$, we get $\det_{H}(\theta^{-1}) = \alpha^{-1}$.
	\end{proof}
	
	By Theorem \ref{t34}, we have ${\Delta_{K}}^{-1} = -\delta^{-1}\gamma{{\Delta}_{H}}^{-1}(-\beta\delta^{-1})+ \delta^{-1}$. This gives an elegant form for the inverse of an automorphisms in $\mathcal{A}$.
	
	\begin{theorem}
		Let $G = H\rtimes K$ and $\theta= \begin{pmatrix}
		\alpha & \beta\\ \gamma & \delta
		\end{pmatrix}\in \mathcal{M}$ such that both $\alpha$ and $\delta$ are invertible. Then, if $\det_{H}(\theta)$(or $\det_{K}(\theta)$) is invertible, then 
		\[\theta^{-1} = \begin{pmatrix}
		{\Delta_{H}}^{-1} & -\alpha^{-1}\beta{\Delta_{K}}^{-1}\\
		-\delta^{-1}\gamma{\Delta_{H}}^{-1}& {\Delta_{K}}^{-1}
		\end{pmatrix}.\]
	\end{theorem}

\section{Another description of $Aut(G)$}
In this section, we will study two equivalent subsets of $\mathcal{M}$ associated with the subgroups $H$ and $K$ of the group $G = H\bowtie K$, and consists of all the elements of $\mathcal{M}$ whose determinants are defined and invertible. 

\begin{definition}
	Let $\mathcal{M}$ be defined as above. Then we define a set  consisting of all the elements of $\mathcal{M}$ whose determinats are defined and invertible, as 
	\[\mathcal{G}_{H} = \{\theta\in \mathcal{M} \mid {\det}_{H}(\theta)\; \text{is defined and invertible}\}. \]
	or equivalently 
	\[\mathcal{G}_{K} = \{\theta\in \mathcal{M} \mid {\det}_{K}(\theta)\; \text{is defined and invertible}\}. \]
\end{definition}
%
	\begin{remark}\label{rem41}
	Let $\theta \in {\mathcal{G}}_{H}$. Then $\det_{H}(\theta)$ is defined and invertible. Therefore, using Theorem \ref{t33}, we have $\theta$ is invertible. Thus $\theta$ is an automorphism of $G$. Hence $\theta\in \mathcal{A}$ and so, ${\mathcal{G}}_{H}\subseteq \mathcal{A}$.
\end{remark}

	Using Example \ref{ex1}, one can easily observe that in general, $\mathcal{G}_{H}\ne \mathcal{A}$. In the following proposition, we will prove that the $\mathcal{G}_{H}$ is corresponds to a subset of $Aut(G)$. 

	\begin{proposition}\label{prop41}
	Let $G = H \bowtie K$. Then $\mathcal{G}_{H}$ is equivalent to a subset of $Aut(G)$.		
\end{proposition}
\begin{proof}
	The proof follows directly from Remark \ref{rem41} and the fact that $\mathcal{A} \simeq Aut(G)$.
\end{proof}
\begin{proposition}\label{pr41}
	Let $\mathcal{G}_{H}$ be defined as above. If for all $\begin{pmatrix}
	\alpha & \beta \\ \gamma & \delta
	\end{pmatrix}\in \mathcal{A}$, the map $\delta$ is invertible, then $\mathcal{G}_{H}\le \mathcal{A}$.
\end{proposition}
\begin{proof}
	Let $\theta$, $\theta^{\prime}\in \mathcal{G}_{H}$. Then both $\det_{H}(\theta)$ and $\det_{H}(\theta^{\prime})$ are defined and invertible. Therefore, by Theorem \ref{t33},  both $\theta$ and $\theta^{\prime}$ are invertible. This implies that $\theta^{\prime}\theta\in \mathcal{A}$ and consequently $\theta^{\prime}\theta \in \mathcal{A}$. Now, using the hypothesis and Theorem \ref{t33}, we get $\det_{H}(\theta^{\prime}\theta)$ is defined and invertible. Thus $\theta^{\prime}\theta\in \mathcal{G}_{H}$ for all $\theta^{\prime}, \theta \in \mathcal{G}_{H}$. 
	
	Also, note that for any $\theta = \begin{pmatrix}
	\alpha & \beta \\ \gamma & \delta
	\end{pmatrix} \in \mathcal{G}_{H}$, $\det_{H}(\theta)$ is defined and invertible. Then Theorem \ref{t33} implies that $\theta$ is invertible. Therefore, $\theta\in \mathcal{A}$ and so $\theta^{-1}\in \mathcal{A}$. Again, using the hypothesis and Theorem \ref{t33}, we have $\theta^{-1}\in \mathcal{G}_{H}$. Hence $\mathcal{G}_{H}\le \mathcal{A}$.
\end{proof}

\begin{corollary}\label{cor1}
	Let $G = H \bowtie K$. If for all $\begin{pmatrix}
	\alpha & \beta \\ \gamma & \delta
	\end{pmatrix}\in \mathcal{A}$, the map $\delta$ is invertible, then $\mathcal{G}_{H} = \mathcal{A}$. Moreover, $\mathcal{G}_{H} \simeq Aut(G)$.
\end{corollary}
\begin{proof}
	Let $\theta = \begin{pmatrix}
	\alpha &\beta\\ \gamma & \delta
	\end{pmatrix}\in \mathcal{A}$ be such that $\delta$ is invertible. Then using Theorem \ref{t33}, we get $\det_{H}(\theta)$ is defined and invertible. This implies that $\theta\in \mathcal{G}_{H}$ and so, $\mathcal{A} \subseteq \mathcal{G}_{H}$. Thus, using Remark \ref{rem41} and Proposition \ref{prop41}, we get $\mathcal{A} = \mathcal{G}_{H}$ as groups. Hence $\mathcal{G}_{H} \simeq Aut(G)$. 
\end{proof}
Now, combining Propositions \ref{t38}--\ref{pr41} and Corollary \ref{cor1}, we have the following theorem.  
\begin{theorem}\label{t38}
	Let $G = H\bowtie K$. Then $Aut(G)\simeq \mathcal{G}_{H}$ if and only if for all $\begin{pmatrix}
	\alpha & \beta \\ \gamma & \delta
	\end{pmatrix}\in \mathcal{A}$, the map $\delta$ is invertible.		
\end{theorem}

	Following the same lines of proofs of Propositions \ref{prop41}--\ref{pr41}, Corollary \ref{cor1} and Theorem \ref{t38}, we have the similar results for $\mathcal{G}_{K}$.

\begin{theorem}\label{th42}
	Let $G = H\bowtie K$. Then
	\begin{itemize}
		\item[$(i)$]  $\mathcal{G}_{K}$ corresponds to a subset of $Aut(G)$.
	\item[$(ii)$] if for all $\begin{pmatrix}
	\alpha & \beta \\ \gamma & \delta
	\end{pmatrix}\in \mathcal{A}$, the map $\alpha$ is invertible, then $\mathcal{G}_{K}\le \mathcal{A}$.		 
	\item[$(iii)$]  if for all $\begin{pmatrix}
	\alpha & \beta \\ \gamma & \delta
	\end{pmatrix}\in \mathcal{M}$, the map $\alpha$ is invertible, then  $Aut(G) \simeq \mathcal{G}_{K}$.
	\item[$(iv)$]  $\mathcal{G}_{K} \simeq Aut(G)$ if and only if for all $\begin{pmatrix}
	\alpha & \beta \\ \gamma & \delta
	\end{pmatrix}\in \mathcal{A}$, the map $\alpha$ is invertible.
	\end{itemize} 
\end{theorem}
	From Theorems \ref{t38} and \ref{th42}, the following theorem is well established.
\begin{theorem}\label{t43}
	Let $G = H \bowtie K$. Then if for all $\begin{pmatrix}
	\alpha & \beta \\ \gamma & \delta
	\end{pmatrix}\in \mathcal{M}$, both $\alpha$ and $\delta$ are invertible, then $\mathcal{G}_{H} = \mathcal{A} = \mathcal{G}_{K}$ and $\mathcal{G}_{H} \simeq Aut(G) \simeq \mathcal{G}_{K}$.
\end{theorem}


\begin{definition}
	Let $\mathcal{G}_{H}$ and $\mathcal{G}_{K}$ be the groups defined as above. Then the set of all the determinants of the members of $\mathcal{G}_{H}$ and $\mathcal{G}_{K}$ are defined respectively as
	\begin{equation*}
	\mathcal{D}_{\mathcal{G}_{H}} = \{{\det}_{H}(\theta) \mid \theta \in \mathcal{G}_{H}\},
	\end{equation*}
	and
	\begin{equation*}
	\mathcal{D}_{\mathcal{G}_{K}} = \{{\det}_{K}(\theta) \mid \theta \in \mathcal{G}_{K}\}.
	\end{equation*}
\end{definition}

Using elementary group theory, one can easily prove the following lemma.
\begin{lemma}\label{l2}
	The set $\mathcal{D}_{\mathcal{G}_{H}}$ is a group with binary operation as the usual composition of maps and the identity element as $I_{H}$.
\end{lemma}

In the following proposition, we prove that $\mathcal{D}_{\mathcal{G}_{H}}$ is a subgroup of $Aut(H)$.

\begin{proposition}\label{p1}
	Let $G = H\bowtie K$ such that $\det_{H}(\theta)$ is a group homomorphism for all $\theta \in \mathcal{G}_{H}$. Then $\mathcal{D}_{\mathcal{G}_{H}}\subseteq Aut(H)$.
\end{proposition}
\begin{proof}
	Let $\theta\in \mathcal{G}_{H}$ such that $\det_{H}(\theta)$ is a group homomorphism. Then $\det_{H}(\theta)\in \mathcal{D}_{\mathcal{G}_{H}}$. Since $\theta \in \mathcal{G}_{H}$, $\det_{H}(\theta)$ is defined and invertible. Therefore, $\det_{H}(\theta)\in Aut(H)$. Hence $\mathcal{D}_{\mathcal{G}_{H}}\subseteq Aut(H)$.
	
\end{proof}
Using the similar arguments as in Lemma \ref{l2} and Proposition \ref{p1}, we have similar results for $\mathcal{D}_{\mathcal{G}_{K}}$ given below.
\begin{proposition}\label{p2}
	Let $G = H\bowtie K$ and $\theta \in \mathcal{G}_{K}$. Then 
	\begin{itemize}
		\item[$(i)$] $\mathcal{D}_{\mathcal{G}_{K}}$ is a group with the usual composition of maps,
		\item[$(ii)$] if $\det_{K}(\theta)$ is a group homomorphism for all $\theta \in \mathcal{G}_{K}$, then  $\mathcal{D}_{\mathcal{G}_{K}}\subseteq Aut(K)$.
	\end{itemize} 
\end{proposition}

\section{Determinant of automorphisms of the Alternating group $A_{5}$}\label{s4}
In this section, we will compute both the determinats of the automorphisms of the alternating group $A_{5}$.

Throughout this section, we will denote $x = (1\; 2)(3\; 4)$, $y = (1\; 2\; 3)$ and $z = (1\; 2\; 3\; 4\; 5)$. Now, let $H = \langle x, y\mid x^{2} = y^{3} = (xy)^{3}\rangle$ and $K = \langle z\mid z^{5} = 1\rangle$. Then $H \simeq A_{4}$, the alternating group of degree 4 and $K \simeq \mathbb{Z}_{5}$. Now, define the mutual actions $\sigma: K\times H \longrightarrow H$ and $\tau: K\times H \longrightarrow K$ as follows:
\begin{align*}
\sigma(z, x) = z\cdot x = xyx, \; \sigma(z, y) = z\cdot y = y^{2}x, \; \tau(z, x) = z^{x} = z^{2},\; \text{and}\; \tau(z, y) = z^{y} = z. 
\end{align*}
Thus we get $H\bowtie K = \langle x,y,z\mid x^{2} = 1 = y^{3} = z^{5} = (xy)^{3}, zx = xyxz^{2}, zy = y^{2}xz\rangle$. One can easily observe that $A_{5} \simeq H \bowtie K$. It is already known that $Aut(A_{5}) \simeq S_{5}$. 

Since the group $H$ has 12 elements, we label the elements of $H$ as follows:
\begin{align*}
&1\mapsto 0,  x \mapsto 1,  y\mapsto 2,  y^{2}\mapsto 3,  xy \mapsto 4,  yx \mapsto 5, xy^{2}\mapsto 6,\\  &y^{2}x \mapsto 7,  xyx\mapsto 8,  yxy\mapsto 9,  yxy^{2} \mapsto 10,  y^{2}xy \mapsto 11.
\end{align*}
Similarly, we label the elements of $K$ as follows: 
\begin{align*}
1 \mapsto 0,  z \mapsto 1,  z^{2} \mapsto 2,  z^{3} \mapsto 3,  z^{4}\mapsto 4.
\end{align*}
Then it is clear that $\mathcal{D}_{\mathcal{G}_{H}} \subseteq S_{12}$ and $\mathcal{D}_{\mathcal{G}_{K}} \subseteq S_{5}$. We will use the above labeling for the cycle decomposition in the following tables. In the following tables, C. D. will denote the cycle decomposition.

 Now, we compute both the determinants for all the automorphisms of $A_{5}$ and represent them as members of their respective symmetric groups in the following Tables \ref{tbl1} -- \ref{tbl14}. Each table corresponds to the cycle type of $H$--determinants. 

All the cycles types of the $H$--determinants corresponding to each automorphism of $A_{5}$ are given in the following table.
\begin{longtable}[H]{|c|c|c|}
			\hline
				\multirow{2}{*}{Sr. No.} & 	\multirow{2}{*}{Cycle Type} & 	\multirow{2}{*}{Number of $\det_{H}$}\\
				& &\\
			\hline
			1 & 1+1+1+1+1+1+1+1+1+1+1+1+1+1+1 & 1\\
			\hline
			2& 2+2+2+2+2+1+1 & 6\\
			\hline
			3 & 2+2+2+2+1+1+1+1& 7\\
			\hline
			4 & 3+3+3+1+1+1 & 8\\
			\hline
			5 & 4+4+2+1+1 & 14\\
			\hline
			6 & 5+5+1+1 & 4\\
			\hline
			7 & 6+4+1+1 & 4\\
			\hline
			8 & 6+3+2+1 & 8\\
			\hline
			9 & 6+3+1+1+1 & 8\\
			\hline
			10 & 6+1+1+1+1+1+1 & 4\\
			\hline
			11 & 7+4+1 & 4\\
			\hline
			12 & 8+3+1 & 4\\
			\hline			
			13 & 8+2+1+1 & 20\\
			\hline
			14 & 10+1+1 & 20\\
			\hline
			15 & 11 + 1 & 8\\
			\hline		
			\caption{Cycle type of $\det_{H}$ corresponding to all the automorphisms of $A_{5}$}	
		\end{longtable}

\begin{table}[H]
	\noindent
	\resizebox{\textwidth}{!}{%
		\begin{tabular}{|c|c|c|c|c|c|c|c|c|}
			\hline
			\multirow{2}{*}{Sr.} & \multicolumn{4}{c|}{$\theta = \begin{pmatrix} \alpha & \beta \\ \gamma & \delta \end{pmatrix}$} & \multicolumn{2}{c|}{$\det_{H}(\theta)$} & \multicolumn{2}{c|}{$\det_{K}(\theta)$}  \\ 
			\cline{2-9}
			No.& $\alpha$ & $\beta$ & $\gamma$ & $\delta$ & $\alpha - \beta\delta^{-1}\gamma$& C. D. &$-\gamma\alpha^{-1}\beta + \delta $& C. D.\\
			\hline 
			\multirow{2}{*}{1} & $x\mapsto yxy^{2}$& \multirow{2}{*}{$z\mapsto y^{2}xy$} & $x \mapsto 1$ & \multirow{2}{*}{$z\mapsto z$} & $x\mapsto yxy^{2} $& \multirow{2}{*}{(1\; 10)(2\; 3)(4\; 9)(5\; 6)(7\; 8)} & \multirow{2}{*}{$z\mapsto z$}&$\multirow{2}{*}{(2 4)}$ \\
			& $y\mapsto y^{2}$ & & $y \mapsto 1$ & & $y\mapsto y^{2}$ & & & \\
			\hline 
			\multirow{2}{*}{2} & $x\mapsto y^{2}xy$& \multirow{2}{*}{$z\mapsto y^{2}x$} & $x \mapsto 1$ & \multirow{2}{*}{$z\mapsto z$} & $x\mapsto y^{2}xy $& \multirow{2}{*}{(1\; 11)(2\; 3)(4\; 7)(5\; 9)(6\; 8)} & \multirow{2}{*}{$z\mapsto z$}&$\multirow{2}{*}{(2 3)}$ \\
			& $y\mapsto y^{2}$ & & $y \mapsto 1$ & & $y\mapsto y^{2}$ & && \\
			\hline
			\multirow{2}{*}{3} & $x\mapsto x$& \multirow{2}{*}{$z\mapsto y$} & $x \mapsto 1$ & \multirow{2}{*}{$z\mapsto z$} & $x\mapsto x$& \multirow{2}{*}{(2\; 3)(4\; 6)(5\; 7)(8\; 9)(10\; 11)} & \multirow{2}{*}{$z\mapsto z$} &$\multirow{2}{*}{(3 4)}$\\
			& $y\mapsto y^{2}$ & & $y \mapsto 1$ & & $y\mapsto y^{2}$ & && \\
\hline
		\end{tabular}%
	}
	\caption{$\det_{H}$ and $\det_{K}$ corresponding to the cycle type $2+2+2+2+2+1+1$ of $\det_{H}$}\label{tbl1}
\end{table}

\begin{table}[H]
\noindent
\resizebox{\textwidth}{!}{%
	\begin{tabular}{|c|c|c|c|c|c|c|c|c|}
		\hline
		\multirow{2}{*}{Sr.} & \multicolumn{4}{c|}{$\theta = \begin{pmatrix} \alpha & \beta \\ \gamma & \delta \end{pmatrix}$} & \multicolumn{2}{c|}{$\det_{H}(\theta)$} & \multicolumn{2}{c|}{$\det_{K}(\theta)$}  \\ 
		\cline{2-9}
		No.& $\alpha$ & $\beta$ & $\gamma$ & $\delta$ & $\alpha - \beta\delta^{-1}\gamma$& C. D. &$-\gamma\alpha^{-1}\beta + \delta $& C. D.\\
		\hline 
			\multirow{2}{*}{1}  & $x\mapsto x$&  \multirow{2}{*}{$z\mapsto yx$} & $x \mapsto 1$ & \multirow{2}{*}{$z\mapsto z^{2}$} & $x\mapsto  x$&\multirow{2}{*}{$(2\; 8)(3\; 9)(4\; 5)(6\; 7)$}&  \multirow{2}{*}{$z\mapsto z^{2}$} & $\multirow{2}{*}{(1 2)(3 4)}$ \\
			 & $y\mapsto xyx$ & & $y \mapsto 1$ & &$y\mapsto xyx$  && &\\
			\hline 
			\multirow{2}{*}{2}  & $x\mapsto x$&  \multirow{2}{*}{$z\mapsto xyx$} & $x \mapsto 1$ & \multirow{2}{*}{$z\mapsto z^{3}$} & $x\mapsto x $& \multirow{2}{*}{$(2\; 5)(3\; 6)(4\; 8)(7\; 9)$}&  \multirow{2}{*}{$z\mapsto z^{3}$} & $\multirow{2}{*}{(1 3)(2 4)}$ \\
			 & $y\mapsto yx$ & & $y \mapsto 1$ &&$y\mapsto yx$  && &\\
			\hline
			\multirow{2}{*}{3}  & $x\mapsto xy^{2}$&  \multirow{2}{*}{$z\mapsto 1$} & $x \mapsto z$ & \multirow{2}{*}{$z\mapsto z^{4}$} & $x\mapsto  xy^{2}$& \multirow{2}{*}{$(1\; 6)(2\; 7)(3\; 11)(8\; 9)$}&  \multirow{2}{*}{$z\mapsto z^{4}$} & $\multirow{2}{*}{(1 4)(2 3)}$ \\
			 & $y\mapsto y^{2}x$ & & $y \mapsto z^{4}$ &&$y\mapsto y^{2}x$  && &\\
\hline
		\end{tabular}%
}
\caption{$\det_{H}$ and $\det_{K}$ corresponding to the cycle type $2+2+2+2+1+1+1+1$ of $\det_{H}$}
\end{table}	
\begin{table}[H]
\noindent
\resizebox{\textwidth}{!}{%
	\begin{tabular}{|c|c|c|c|c|c|c|c|c|}
		\hline
		\multirow{2}{*}{Sr.} & \multicolumn{4}{c|}{$\theta = \begin{pmatrix} \alpha & \beta \\ \gamma & \delta \end{pmatrix}$} & \multicolumn{2}{c|}{$\det_{H}(\theta)$} & \multicolumn{2}{c|}{$\det_{K}(\theta)$}  \\ 
		\cline{2-9}
		No.& $\alpha$ & $\beta$ & $\gamma$ & $\delta$ & $\alpha - \beta\delta^{-1}\gamma$& C. D. &$-\gamma\alpha^{-1}\beta + \delta $& C. D.\\
		\hline
	\multirow{2}{*}{1}  & $x\mapsto yxy^{2}$&  \multirow{2}{*}{$z\mapsto yxy$} & $x \mapsto 1$ & \multirow{2}{*}{$z\mapsto z$} & $x\mapsto  yxy^{2}$& \multirow{2}{*}{$(1\; 10 \; 11)(4\; 5\; 8)(6\; 9\; 7)$}& \multirow{2}{*}{$z\mapsto z$} & $\multirow{2}{*}{(2 4 3)}$\\
	& $y\mapsto y$ & & $y \mapsto 1$ & &$y\mapsto$  y& &&\\
	\hline 
	\multirow{2}{*}{2}  & $x\mapsto y^{2}xy$&  \multirow{2}{*}{$z\mapsto yx$} & $x \mapsto 1$ & \multirow{2}{*}{$z\mapsto z$}& $x\mapsto y^{2}xy $& \multirow{2}{*}{$(1\; 11 \; 10)(4\; 8\; 5)(6\; 7\; 9)$}& \multirow{2}{*}{$z\mapsto z$}& $\multirow{2}{*}{(2 3 4)}$\\
	& $y\mapsto y$ & & $y \mapsto 1$ &&$y\mapsto y$  &&& \\
	\hline
	\multirow{2}{*}{3}  & $x\mapsto yxy^{2}$&  \multirow{2}{*}{$z\mapsto x$} & $x \mapsto 1$ & \multirow{2}{*}{$z\mapsto z^{2}$} & $x\mapsto yxy^{2} $& \multirow{2}{*}{$(1\; 10 \; 11)(2\; 8\; 5)(3\; 6\; 9)$}& \multirow{2}{*}{$z\mapsto z^{2}$}& $\multirow{2}{*}{(1 2 3)}$\\
	& $y\mapsto xyx$ & & $y \mapsto 1$ &&$y\mapsto xyx$  &&& \\
\hline
\end{tabular}%
}
\caption{$\det_{H}$ and $\det_{K}$ corresponding to the cycle type $3+3+3+1+1+1$ of $\det_{H}$}
\end{table}
\begin{table}[H]
	\noindent
	\resizebox{\textwidth}{!}{%
		\begin{tabular}{|c|c|c|c|c|c|c|c|c|}
			\hline
			\multirow{2}{*}{Sr.} & \multicolumn{4}{c|}{$\theta = \begin{pmatrix} \alpha & \beta \\ \gamma & \delta \end{pmatrix}$} & \multicolumn{2}{c|}{$\det_{H}(\theta)$} & \multicolumn{2}{c|}{$\det_{K}(\theta)$}  \\ 
			\cline{2-9}
			No.& $\alpha$ & $\beta$ & $\gamma$ & $\delta$ & $\alpha - \beta\delta^{-1}\gamma$& C. D. &$-\gamma\alpha^{-1}\beta + \delta $& C. D.\\
			\hline
			\multirow{2}{*}{1}  & $x\mapsto yxy^{2}$&  \multirow{2}{*}{$z\mapsto xy$} & $x \mapsto 1$ & \multirow{2}{*}{$z\mapsto z^{2}$}& $x\mapsto yxy^{2} $& \multirow{2}{*}{$(2\; 9\; 5\; 7)(3\; 8\; 6\; 4)(1\; 10)$} & \multirow{2}{*}{$z\mapsto z^{2}$} & $\multirow{2}{*}{(1 2 3 4)}$ \\
			& $y\mapsto yxy$ & & $y \mapsto 1$ & &$y\mapsto yxy$ & &&\\
			\hline 
			\multirow{2}{*}{2}  & $x\mapsto y^{2}xy$&  \multirow{2}{*}{$z\mapsto 1$} & $x \mapsto 1$ & \multirow{2}{*}{$z\mapsto z^{2}$}& $x\mapsto y^{2}xy $& \multirow{2}{*}{$(2\; 9\; 4\; 6)(3\; 8\; 7\; 5)(1\; 11)$} & \multirow{2}{*}{$z\mapsto z^{2}$} &$\multirow{2}{*}{(1 2 4 3)}$ \\
			& $y\mapsto yxy$ & & $y \mapsto 1$ &&$y\mapsto yxy$  &&& \\
			\hline
			\multirow{2}{*}{3}  & $x\mapsto xy$&  \multirow{2}{*}{$z\mapsto 1$} & $x \mapsto z^{3}$ & \multirow{2}{*}{$z\mapsto z^{2}$} & $x\mapsto xy $& \multirow{2}{*}{$(1\; 4\; 9\; 3)(5\; 7\; 6\; 11)(2\; 8)$} & \multirow{2}{*}{$z\mapsto z^{2}$}&$\multirow{2}{*}{(1 2 4 3)}$ \\
			& $y\mapsto xyx$ & & $y \mapsto z^{2}$ &&$y\mapsto xyx$  &&&\\
\hline
		\end{tabular}%
	}
	\caption{$\det_{H}$ and $\det_{K}$ corresponding to the cycle type $4+4+2+1+1$ of $\det_{H}$}
\end{table}
\begin{table}[H]
\noindent
\resizebox{\textwidth}{!}{%
	\begin{tabular}{|c|c|c|c|c|c|c|c|c|}
		\hline
		\multirow{2}{*}{Sr.} & \multicolumn{4}{c|}{$\theta = \begin{pmatrix} \alpha & \beta \\ \gamma & \delta \end{pmatrix}$} & \multicolumn{2}{c|}{$\det_{H}(\theta)$} & \multicolumn{2}{c|}{$\det_{K}(\theta)$}  \\ 
		\cline{2-9}
		No.& $\alpha$ & $\beta$ & $\gamma$ & $\delta$ & $\alpha - \beta\delta^{-1}\gamma$& C. D. &$-\gamma\alpha^{-1}\beta + \delta $& C. D.\\
		\hline
	\multirow{2}{*}{1}  & $x\mapsto yxy$&  \multirow{2}{*}{$z\mapsto 1$} & $x \mapsto z^{4}$ & \multirow{2}{*}{$z\mapsto z$} & $x\mapsto yxy $& \multirow{2}{*}{$(1\; 9\; 5\; 8\; 6)(2\; 3\; 11\; 7\; 4)$} & \multirow{2}{*}{$z\mapsto z$} & \multirow{2}{*}{$e_{S_{5}}$}\\
	& $y\mapsto y^{2}$ & & $y \mapsto z$ &&$y\mapsto y^{2}$  && &\\
	\hline 
	\multirow{2}{*}{2}  & $x\mapsto xy^{2}$&  \multirow{2}{*}{$z\mapsto 1$} & $x \mapsto z$ & \multirow{2}{*}{$z\mapsto z$} & $x\mapsto xy^{2} $& \multirow{2}{*}{$(1\; 6\; 8\; 5\; 9)(2\; 4\; 7\; 11\; 3)$} & \multirow{2}{*}{$z\mapsto z$} & \multirow{2}{*}{$e_{S_{5}}$}\\
	& $y\mapsto xy$ & & $y \mapsto z$ &&$y\mapsto xy$  &&& \\
	\hline
	\multirow{2}{*}{3} & $x\mapsto yx$&  \multirow{2}{*}{$z\mapsto 1$} & $x \mapsto z^{4}$ & \multirow{2}{*}{$z\mapsto z$} & $x\mapsto yx $& \multirow{2}{*}{$(1\; 5\; 6\; 9\; 8)(2\; 11\; 4\; 3\; 7)$} & \multirow{2}{*}{$z\mapsto z$}& \multirow{2}{*}{$e_{S_{5}}$} \\
	& $y\mapsto y^{2}xy$ & & $y \mapsto z^{3}$ &&$y\mapsto y^{2}xy$  &&& \\
	\hline
	\multirow{2}{*}{4}  & $x\mapsto xyx$&  \multirow{2}{*}{$z\mapsto 1$} & $x \mapsto z$ & \multirow{2}{*}{$z\mapsto z$} & $x\mapsto xyx $& \multirow{2}{*}{$(1\; 8\; 9\; 6\; 5)(2\; 7\; 3\; 4\; 11)$} & \multirow{2}{*}{$z\mapsto z$}& \multirow{2}{*}{$e_{S_{5}}$}\\
	&  $y\mapsto y^{2}x$ & & $y \mapsto 1$ &&$y\mapsto y^{2}x$  &&& \\
	\hline
\end{tabular}%
}
\caption{$\det_{H}$ and $\det_{K}$ corresponding to the cycle type $5+5+1+1$ of $\det_{H}$}
\end{table}
		
\begin{table}[H]
\noindent
\resizebox{\textwidth}{!}{%
	\begin{tabular}{|c|c|c|c|c|c|c|c|c|}
		\hline
		\multirow{2}{*}{Sr.} & \multicolumn{4}{c|}{$\theta = \begin{pmatrix} \alpha & \beta \\ \gamma & \delta \end{pmatrix}$} & \multicolumn{2}{c|}{$\det_{H}(\theta)$} & \multicolumn{2}{c|}{$\det_{K}(\theta)$}  \\ 
		\cline{2-9}
		No.& $\alpha$ & $\beta$ & $\gamma$ & $\delta$ & $\alpha - \beta\delta^{-1}\gamma$& C. D. &$-\gamma\alpha^{-1}\beta + \delta $& C. D.\\
		\hline
	\multirow{2}{*}{1} & $x\mapsto yxy^{2}$&  \multirow{2}{*}{$z\mapsto yx$} & $x \mapsto z^{3}$ & \multirow{2}{*}{$z\mapsto z^{2}$}& $x\mapsto y^{2} $& \multirow{2}{*}{$(2\; 9\; 4\; 8\; 7\; 5)(1\; 3\; 11\; 6)$} & \multirow{2}{*}{$z\mapsto z^{4}$} & \multirow{2}{*}{(1 4 2)} \\
	& $y\mapsto y^{2}xy$ & & $y \mapsto z^{2}$ & &$y\mapsto yxy$  &&&\\
	\hline 
	\multirow{2}{*}{2} & $x\mapsto yxy^{2}$&  \multirow{2}{*}{$z\mapsto yx$} & $x \mapsto z^{2}$ & \multirow{2}{*}{$z\mapsto z^{2}$} & $x\mapsto y^{2}x $&  \multirow{2}{*}{$(2\; 5\; 3\; 8\; 4\; 6)(1\; 7\; 11\; 9)$} & \multirow{2}{*}{$z\mapsto z^{3}$} & \multirow{2}{*}{(1 3 4)}\\
	& $y\mapsto xy^{2}$ & & $y \mapsto z^{2}$ &&$y\mapsto yx$  &&& \\
	\hline
	\multirow{2}{*}{3}  & $x\mapsto y^{2}$&  \multirow{2}{*}{$z\mapsto xyx$} & $x \mapsto z^{3}$ & \multirow{2}{*}{$z\mapsto z^{3}$}& $x\mapsto xy $& \multirow{2}{*}{$(2\; 9\; 3\; 8\; 7\; 6)(1\; 4\; 11\; 5)$} & \multirow{2}{*}{$z\mapsto z^{3}$}& \multirow{2}{*}{(1 3 2)}  \\
	& $y\mapsto yxy$ & & $y \mapsto 1$ & &$y\mapsto yxy$  &&&\\
	\hline
	\multirow{2}{*}{4} & $x\mapsto y^{2}x$&  \multirow{2}{*}{$z\mapsto xyx$} & $x \mapsto z^{2}$ & \multirow{2}{*}{$z\mapsto z^{3}$}& $x\mapsto y $& \multirow{2}{*}{$(3\; 9\; 4\; 6\; 7\; 5)(1\; 2\; 11\; 8)$} & \multirow{2}{*}{$z\mapsto z$} & \multirow{2}{*}{(2 3 4)} \\
	& $y\mapsto y^{2}$ & & $y \mapsto z$ &&$y\mapsto y^{2}xy$  & &&\\
	\hline
\end{tabular}%
}
\caption{$\det_{H}$ and $\det_{K}$ corresponding to the cycle type $6+4+1+1$ of $\det_{H}$}
\end{table}

	
\begin{table}[H]
	\noindent
	\resizebox{\textwidth}{!}{%
		\begin{tabular}{|c|c|c|c|c|c|c|c|c|}
			\hline
			\multirow{2}{*}{Sr.} & \multicolumn{4}{c|}{$\theta = \begin{pmatrix} \alpha & \beta \\ \gamma & \delta \end{pmatrix}$} & \multicolumn{2}{c|}{$\det_{H}(\theta)$} & \multicolumn{2}{c|}{$\det_{K}(\theta)$}  \\ 
			\cline{2-9}
			No.& $\alpha$ & $\beta$ & $\gamma$ & $\delta$ & $\alpha - \beta\delta^{-1}\gamma$& C. D. &$-\gamma\alpha^{-1}\beta + \delta $& C. D.\\
			\hline
			\multirow{2}{*}{1} & $x\mapsto y^{2}$&  \multirow{2}{*}{$z\mapsto yxy$} & $x \mapsto z^{3}$ & \multirow{2}{*}{$z\mapsto z$}& $x\mapsto y^{2}x $& \multirow{2}{*}{$(1\; 7\; 5\; 4\; 2\; 3)(6\; 10\; 11)(8\; 9)$}& \multirow{2}{*}{$z\mapsto z$} & \multirow{2}{*}{(2 3 4)} \\
			&  $y\mapsto xy$ & & $y \mapsto z$ &&$y\mapsto y^{2}$  && &\\
			\hline 
			\multirow{2}{*}{2} & $x\mapsto yxy^{2}$&  \multirow{2}{*}{$z\mapsto yx$} & $x \mapsto z$ & \multirow{2}{*}{$z\mapsto z$}& $x\mapsto y^{2}x $& \multirow{2}{*}{$(2\; 8\; 4\; 9\; 11\; 3)(1\; 7\; 10)(5\; 6)$}& \multirow{2}{*}{$z\mapsto z$} & \multirow{2}{*}{(2 4 3)}\\
			&  $y\mapsto y^{2}$ & & $y \mapsto z$ &&$y\mapsto xyx$  & &&\\
			\hline                               
			\multirow{2}{*}{3}  & $x\mapsto y^{2}x$&  \multirow{2}{*}{$z\mapsto x$} & $x \mapsto z^{2}$ & \multirow{2}{*}{$z\mapsto z^{2}$}& $x\mapsto y^{2} $& \multirow{2}{*}{$(1\; 3\; 8\; 2\; 4\; 7)(9\; 10\; 11)(5\; 6)$}& \multirow{2}{*}{$z\mapsto z^{3}$} & \multirow{2}{*}{(1 3 2)} \\
			&  $y\mapsto xyx$ & & $y \mapsto z^{2}$ &&$y\mapsto xy$  && &\\
\hline			
		\end{tabular}%
	}
	\caption{$\det_{H}$ and $\det_{K}$ corresponding to the cycle type $6+3+2+1$ of $\det_{H}$}
\end{table}
\begin{table}[H]
	\noindent
	\resizebox{\textwidth}{!}{%
		\begin{tabular}{|c|c|c|c|c|c|c|c|c|}
			\hline
			\multirow{2}{*}{Sr.} & \multicolumn{4}{c|}{$\theta = \begin{pmatrix} \alpha & \beta \\ \gamma & \delta \end{pmatrix}$} & \multicolumn{2}{c|}{$\det_{H}(\theta)$} & \multicolumn{2}{c|}{$\det_{K}(\theta)$}  \\ 
			\cline{2-9}
			No.& $\alpha$ & $\beta$ & $\gamma$ & $\delta$ & $\alpha - \beta\delta^{-1}\gamma$& C. D. &$-\gamma\alpha^{-1}\beta + \delta $& C. D.\\
			\hline
			\multirow{2}{*}{1} & $x\mapsto y^{2}x$& \multirow{2}{*}{$z\mapsto xy$} & $x \mapsto z^{2}$ & \multirow{2}{*}{$z\mapsto z^{2}$} & $x\mapsto xy $& \multirow{2}{*}{(1\; 4\; 2\; 3\; 9\; 10)(8\; 7\; 11)} & \multirow{2}{*}{$z\mapsto z$} & \multirow{2}{*}{(2 3)} \\
			& $y\mapsto y^{2}xy$ & & $y \mapsto z^{2}$ & &$y\mapsto y^{2}$ && &\\
			\hline 
			\multirow{2}{*}{2} & $x\mapsto yxy^{2}$& \multirow{2}{*}{$z\mapsto xy$} & $x \mapsto z^{2}$ & \multirow{2}{*}{$z\mapsto z^{2}$} & $x\mapsto xy^{2} $& \multirow{2}{*}{(1\; 6\; 7\; 8\; 5\; 10)(4\; 11\; 9)} & \multirow{2}{*}{$z\mapsto z^{4}$} & \multirow{2}{*}{(1 4 3 2)} \\
			& $y\mapsto y$ & & $y \mapsto 1$ & &$y\mapsto y$ && &\\
			\hline
			\multirow{2}{*}{3} & $x\mapsto xy$& \multirow{2}{*}{$z\mapsto y^{2}x$} & $x \mapsto z^{3}$ & \multirow{2}{*}{$z\mapsto z^{3}$}& $x\mapsto y^{2}x $& \multirow{2}{*}{(4\; 10\; 11\; 6\; 9\; 8)(1\; 7\; 5)} & \multirow{2}{*}{$z\mapsto z^{4}$} & \multirow{2}{*}{(1 4)}\\
			& $y\mapsto y$ & & $y \mapsto 1$ & &$y\mapsto y$ & &&\\
\hline			
		\end{tabular}%
	}
	\caption{$\det_{H}$ and $\det_{K}$ corresponding to the cycle type $6+3+1+1+1$ of $\det_{H}$}
\end{table}
	\begin{table}
\noindent
\resizebox{\textwidth}{!}{%
	\begin{tabular}{|c|c|c|c|c|c|c|c|c|}
		\hline
		\multirow{2}{*}{Sr.} & \multicolumn{4}{c|}{$\theta = \begin{pmatrix} \alpha & \beta \\ \gamma & \delta \end{pmatrix}$} & \multicolumn{2}{c|}{$\det_{H}(\theta)$} & \multicolumn{2}{c|}{$\det_{K}(\theta)$}  \\ 
		\cline{2-9}
		No.& $\alpha$ & $\beta$ & $\gamma$ & $\delta$ & $\alpha - \beta\delta^{-1}\gamma$& C. D. &$-\gamma\alpha^{-1}\beta + \delta $& C. D.\\
		\hline
	\multirow{2}{*}{1} & $x\mapsto yxy^{2}$& \multirow{2}{*}{$z\mapsto y^{2}x$} & $x \mapsto z^{4}$ & \multirow{2}{*}{$z\mapsto z$}& $x\mapsto yx $& \multirow{2}{*}{(1\; 5\; 6\; 11\; 4\; 3)} & \multirow{2}{*}{$z\mapsto z^{2}$} & \multirow{2}{*}{(1 2 3 4)} \\
	& $y\mapsto x$ & & $y \mapsto z$ & & $y\mapsto y$ & &&\\
	\hline 
	\multirow{2}{*}{2} & $x\mapsto yxy^{2}$& \multirow{2}{*}{$z\mapsto y^{2}x$} & $x \mapsto z$ & \multirow{2}{*}{$z\mapsto z$}& $x\mapsto xyx $& \multirow{2}{*}{(1\; 8\; 9\; 11\; 2\; 7)} & \multirow{2}{*}{$z\mapsto z^{4}$} & \multirow{2}{*}{(1 4 3 2)} \\
	& $y\mapsto xy$ & & $y \mapsto z$ & & $y\mapsto y^{2}x$ & &&\\
	\hline
	\multirow{2}{*}{3} & $x\mapsto yx$& \multirow{2}{*}{$z\mapsto y^{2}$} & $x \mapsto z^{4}$ & \multirow{2}{*}{$z\mapsto z^{4}$} & $x\mapsto yxy $& \multirow{2}{*}{(1\; 9\; 5\; 11\; 7\; 4)} & \multirow{2}{*}{$z\mapsto z^{4}$} & \multirow{2}{*}{(1 4 2 3)} \\
	& $y\mapsto y$ & & $y \mapsto 1$ & & $y\mapsto y$ & &&\\
	\hline
	\multirow{2}{*}{4} & $x\mapsto xyx$& \multirow{2}{*}{$z\mapsto y^{2}$} & $x \mapsto z$ & \multirow{2}{*}{$z\mapsto z^{4}$} & $x\mapsto xy^{2} $& \multirow{2}{*}{(1\; 6\; 8\; 11\; 3\; 2)} & \multirow{2}{*}{$z\mapsto z^{3}$} & \multirow{2}{*}{(1 3 2 4)} \\
	& $y\mapsto yx$ & & $y \mapsto z^{3}$ & & $y\mapsto x$ & &&\\
	\hline
\end{tabular}%
}
\caption{$\det_{H}$ and $\det_{K}$ corresponding to the cycle type $6+1+1+1+1+1+1$ of $\det_{H}$}
\end{table}
	\begin{table}
	\noindent
	\resizebox{\textwidth}{!}{%
		\begin{tabular}{|c|c|c|c|c|c|c|c|c|}
			\hline
			\multirow{2}{*}{Sr.} & \multicolumn{4}{c|}{$\theta = \begin{pmatrix} \alpha & \beta \\ \gamma & \delta \end{pmatrix}$} & \multicolumn{2}{c|}{$\det_{H}(\theta)$} & \multicolumn{2}{c|}{$\det_{K}(\theta)$}  \\ 
			\cline{2-9}
			No.& $\alpha$ & $\beta$ & $\gamma$ & $\delta$ & $\alpha - \beta\delta^{-1}\gamma$& C. D. &$-\gamma\alpha^{-1}\beta + \delta $& C. D.\\
			\hline
	\multirow{2}{*}{1} & $x\mapsto xyx$& \multirow{2}{*}{$z\mapsto xy$} & $x \mapsto z$ & \multirow{2}{*}{$z\mapsto z^{2}$} & $x\mapsto yx $& \multirow{2}{*}{(2\; 8\; 4\; 9\; 7\; 3\; 11)(1\; 5\; 6\; 10)} & \multirow{2}{*}{$z\mapsto z^{2}$} & \multirow{2}{*}{(1 2)}\\
	& $y\mapsto xy^{2}$ & & $y \mapsto z^{2}$ & & $y\mapsto xyx$ & &&  \\
	\hline 
	
	\multirow{2}{*}{2} & $x\mapsto yxy^{2}$& \multirow{2}{*}{$z\mapsto y^{2}x$} & $x \mapsto z$ & \multirow{2}{*}{$z\mapsto z^{3}$}& $x\mapsto yxy $& \multirow{2}{*}{(1\; 9\; 7\; 6\; 4\; 5\; 8)(2\; 10\; 11\; 3)} & \multirow{2}{*}{$z\mapsto z$}  & \multirow{2}{*}{(2 4)}\\
	& $y\mapsto y^{2}xy$ & & $y \mapsto z^{2}$ & & $y\mapsto yxy^{2}$ & && \\
	\hline
	
	\multirow{2}{*}{3} & $x\mapsto yx$& \multirow{2}{*}{$z\mapsto y^{2}xy$} & $x \mapsto z^{4}$ & \multirow{2}{*}{$z\mapsto z^{4}$} & $x\mapsto xyx $& \multirow{2}{*}{(2\; 6\; 3\; 7\; 11\; 4\; 5)(1\; 8\; 9\; 10)} & \multirow{2}{*}{$z\mapsto z$} & \multirow{2}{*}{(3 4)} \\
	& $y\mapsto y^{2}x$ & & $y \mapsto z^{4}$ & & $y\mapsto xy^{2}$ && & \\
	\hline
	
	\multirow{2}{*}{4} & $x\mapsto yxy^{2}$& \multirow{2}{*}{$z\mapsto xy$} & $x \mapsto z^{4}$ & \multirow{2}{*}{$z\mapsto z^{4}$} & $x\mapsto xy^{2} $& \multirow{2}{*}{(1\; 6\; 3\; 9\; 2\; 8\; 5)(4\; 10 \; 11\; 7)} & \multirow{2}{*}{$z\mapsto z^{3}$}  & \multirow{2}{*}{(1 3)}\\
	& $y\mapsto xyx$ & & $y \mapsto 1$ && $y\mapsto xyx$ && &\\
	 \hline 
	\end{tabular}%
}
	\caption{$\det_{H}$ and $\det_{K}$ corresponding to the cycle type $7+4+1$ of $\det_{H}$}
\end{table}

	\begin{table}
	\noindent
	\resizebox{\textwidth}{!}{%
		\begin{tabular}{|c|c|c|c|c|c|c|c|c|}
			\hline
			\multirow{2}{*}{Sr.} & \multicolumn{4}{c|}{$\theta = \begin{pmatrix} \alpha & \beta \\ \gamma & \delta \end{pmatrix}$} & \multicolumn{2}{c|}{$\det_{H}(\theta)$} & \multicolumn{2}{c|}{$\det_{K}(\theta)$}  \\ 
			\cline{2-9}
			No.& $\alpha$ & $\beta$ & $\gamma$ & $\delta$ & $\alpha - \beta\delta^{-1}\gamma$& C. D. &$-\gamma\alpha^{-1}\beta + \delta $& C. D.\\
			\hline
	\multirow{2}{*}{1} & $x\mapsto xy^{2}$& \multirow{2}{*}{$z\mapsto xy$} & $x \mapsto z$ & \multirow{2}{*}{$z\mapsto z^{2}$} & $x\mapsto y^{2}x $& \multirow{2}{*}{(1\; 7\; 9\; 6\; 5\; 4\; 8\; 10)(11\; 3\; 2)} & \multirow{2}{*}{$z\mapsto z$} & \multirow{2}{*}{(3 4)} \\
	& $y\mapsto y$ & & $y \mapsto z^{4}$ & & $y\mapsto y^{2}xy$ & && \\
	\hline 
	
	\multirow{2}{*}{2} & $x\mapsto yxy^{2}$& \multirow{2}{*}{$z\mapsto y^{2}x$} & $x \mapsto z^{3}$ & \multirow{2}{*}{$z\mapsto z^{3}$}& $x\mapsto xyx $& \multirow{2}{*}{(2\; 3\; 6\; 7\; 10\; 11\; 5\; 4)(1\; 8\; 9)} & \multirow{2}{*}{$z\mapsto z^{3}$} & \multirow{2}{*}{(1 3)}  \\
	& $y\mapsto yx$ & & $y \mapsto z^{3}$ & & $y\mapsto y^{2}$ & && \\
	\hline
	
	\multirow{2}{*}{3} & $x\mapsto yxy$& \multirow{2}{*}{$z\mapsto y^{2}xy$} & $x \mapsto z^{4}$ & \multirow{2}{*}{$z\mapsto z^{4}$} & $x\mapsto y^{2} $& \multirow{2}{*}{(1\; 3\; 6\; 9\; 8\; 2\; 5\; 10)(4\; 11\; 7)} & \multirow{2}{*}{$z\mapsto z^{2}$} & \multirow{2}{*}{(1 2)}  \\
	& $y\mapsto yx$ & & $y \mapsto 1$ & & $y\mapsto yx$ & && \\
	\hline
	
	\multirow{2}{*}{4} & $x\mapsto yxy^{2}$& \multirow{2}{*}{$z\mapsto xy$} & $x \mapsto z^{2}$ & \multirow{2}{*}{$z\mapsto z^{4}$} & $x\mapsto yx $& \multirow{2}{*}{(2\; 4\; 7\; 9\; 3\; 10\; 11\; 8)(1\; 5 \; 6)} & \multirow{2}{*}{$z\mapsto z$} & \multirow{2}{*}{(2 4)} \\
	& $y\mapsto y^{2}x$ & & $y \mapsto z^{4}$ & & $y\mapsto xy$ & && \\
	\hline	
	\end{tabular}%
}
\caption{$\det_{H}$ and $\det_{K}$ corresponding to the cycle type $8+3+1$ of $\det_{H}$}
\end{table}

	\begin{table}
		\noindent
		\resizebox{\textwidth}{!}{%
			\begin{tabular}{|c|c|c|c|c|c|c|c|c|}
				\hline
				\multirow{2}{*}{Sr.} & \multicolumn{4}{c|}{$\theta = \begin{pmatrix} \alpha & \beta \\ \gamma & \delta \end{pmatrix}$} & \multicolumn{2}{c|}{$\det_{H}(\theta)$} & \multicolumn{2}{c|}{$\det_{K}(\theta)$}  \\ 
				\cline{2-9}
				No.& $\alpha$ & $\beta$ & $\gamma$ & $\delta$ & $\alpha - \beta\delta^{-1}\gamma$& C. D. &$-\gamma\alpha^{-1}\beta + \delta $& C. D.\\
				\hline
				\multirow{2}{*}{1} & $x\mapsto y$& \multirow{2}{*}{$z\mapsto yx$} & $x \mapsto z^{2}$ & \multirow{2}{*}{$z\mapsto z$} & $x\mapsto xy^{2} $& \multirow{2}{*}{(1\; 6\; 11\; 9\; 4\; 2\; 3\; 10)(7\; 8)} & \multirow{2}{*}{$z\mapsto z^{3}$} & \multirow{2}{*}{(1 3 2)}\\
				& $y\mapsto x$ & & $y \mapsto z$ & & $y\mapsto y^{2}$ & && \\
				\hline 
				
				\multirow{2}{*}{2} & $x\mapsto xyx$& \multirow{2}{*}{$z\mapsto yxy$} & $x \mapsto z$ & \multirow{2}{*}{$z\mapsto z$}& $x\mapsto yxy $& \multirow{2}{*}{(1\; 9\; 10\; 11\; 3\; 2\; 5\; 7)(4\; 6)} & \multirow{2}{*}{$z\mapsto z^{3}$}& \multirow{2}{*}{(1 3)(2 4)}\\
				& $y\mapsto x$ & & $y \mapsto z$ & & $y\mapsto yx$ & & \\
				\hline
				
				\multirow{2}{*}{3} & $x\mapsto yxy^{2}$& \multirow{2}{*}{$z\mapsto yxy$} & $x \mapsto z$ & \multirow{2}{*}{$z\mapsto z$} & $x\mapsto xy $& \multirow{2}{*}{(2\; 6\; 8\; 7\; 3\; 5\; 10\; 11)(1\; 4)} & \multirow{2}{*}{$z\mapsto z^{2}$} & \multirow{2}{*}{(1 2 3)}\\
				& $y\mapsto xy^{2}$ & & $y \mapsto 1$ & & $y\mapsto xy^{2}$ & &&\\
\hline	
			\end{tabular}%
		}
		\caption{$\det_{H}$ and $\det_{K}$ corresponding to the cycle type $8+2+1+1$ of $\det_{H}$}
	\end{table}

	\begin{table}
		\noindent
		\resizebox{\textwidth}{!}{%
			\begin{tabular}{|c|c|c|c|c|c|c|c|c|}
				\hline
				\multirow{2}{*}{Sr.} & \multicolumn{4}{c|}{$\theta = \begin{pmatrix} \alpha & \beta \\ \gamma & \delta \end{pmatrix}$} & \multicolumn{2}{c|}{$\det_{H}(\theta)$} & \multicolumn{2}{c|}{$\det_{K}(\theta)$}  \\ 
				\cline{2-9}
				No.& $\alpha$ & $\beta$ & $\gamma$ & $\delta$ & $\alpha - \beta\delta^{-1}\gamma$& C. D. &$-\gamma\alpha^{-1}\beta + \delta $& C. D.\\
				\hline
				\multirow{2}{*}{1} & $x\mapsto yxy^{2}$& \multirow{2}{*}{$z\mapsto y$} & $x \mapsto z$ & \multirow{2}{*}{$z\mapsto z$}& $x\mapsto yxy $& \multirow{2}{*}{(1\; 9\; 4\; 10\; 11\; 3\; 6\; 7\; 2\; 5)} & \multirow{2}{*}{$z\mapsto z^{2}$} & \multirow{2}{*}{(1 2 3 4)}\\
				& $y\mapsto yx$ & & $y \mapsto 1$ & & $y\mapsto yx$ & && \\
				\hline 
				\multirow{2}{*}{2} & $x\mapsto xy^{2}$& \multirow{2}{*}{$z\mapsto y^{2}xy$} & $x \mapsto z$ & \multirow{2}{*}{$z\mapsto z$} & $x\mapsto y^{2}x $& \multirow{2}{*}{(1\; 7\; 2\; 8\; 4\; 11\; 3\; 9\; 5\; 10)} & \multirow{2}{*}{$z\mapsto z^{3}$} & \multirow{2}{*}{(1 3 2 4)}\\
				& $y\mapsto xyx$ & & $y \mapsto 1$ & & $y\mapsto xyx$ & && \\
				\hline
				
				\multirow{2}{*}{3} & $x\mapsto xyx$& \multirow{2}{*}{$z\mapsto y^{2}x$} & $x \mapsto z$ & \multirow{2}{*}{$z\mapsto z$} & $x\mapsto xy^{2} $& \multirow{2}{*}{(1\; 6\; 9\; 8\; 5\; 11\; 3\; 7\; 2\; 4)} & \multirow{2}{*}{$z\mapsto z$} & \multirow{2}{*}{(3 4)}\\
				& $y\mapsto xy$ & & $y \mapsto 1$ & & $y\mapsto xy$ & && \\
\hline	
			\end{tabular}%
		}
		\caption{$\det_{H}$ and $\det_{K}$ corresponding to the cycle type $10+1+1$ of $\det_{H}$}
	\end{table}

\begin{table}
	\noindent
	\resizebox{\textwidth}{!}{%
		\begin{tabular}{|c|c|c|c|c|c|c|c|c|}
			\hline
			\multirow{2}{*}{Sr.} & \multicolumn{4}{c|}{$\theta = \begin{pmatrix} \alpha & \beta \\ \gamma & \delta \end{pmatrix}$} & \multicolumn{2}{c|}{$\det_{H}(\theta)$} & \multicolumn{2}{c|}{$\det_{K}(\theta)$}  \\ 
			\cline{2-9}
			No.& $\alpha$ & $\beta$ & $\gamma$ & $\delta$ & $\alpha - \beta\delta^{-1}\gamma$& C. D. &$-\gamma\alpha^{-1}\beta + \delta $& C. D.\\
			\hline
			\multirow{2}{*}{1} & $x\mapsto xy$& \multirow{2}{*}{$z\mapsto yxy$} & $x \mapsto z^{3}$ & \multirow{2}{*}{$z\mapsto z$}& $x\mapsto xyx $& \multirow{2}{*}{(1\; 8\; 10\; 11\; 5\; 9\; 4\; 3\; 6\; 7\; 2)} & \multirow{2}{*}{$z\mapsto z^{3}$} & \multirow{2}{*}{(1 3 2)}\\
			& $y\mapsto xy^{2}$ & & $y \mapsto z^{2}$ & & $y\mapsto x$ && & \\
			\hline 
			
			\multirow{2}{*}{2} & $x\mapsto xy^{2}$& \multirow{2}{*}{$z\mapsto yx$} & $x \mapsto z$ & \multirow{2}{*}{$z\mapsto z$} & $x\mapsto yx $& \multirow{2}{*}{(1\; 5\; 11\; 2\; 9\; 7\; 3\; 8\; 6\; 4\; 10)} & \multirow{2}{*}{$z\mapsto z^{3}$} & \multirow{2}{*}{(1 3)(2 4)}\\
			& $y\mapsto yxy$ & & $y \mapsto 1$ & & $y\mapsto yxy$ & && \\
			\hline
			
			\multirow{2}{*}{3} & $x\mapsto yxy^{2}$& \multirow{2}{*}{$z\mapsto yxy$} & $x \mapsto z$ & \multirow{2}{*}{$z\mapsto z^{2}$} & $x\mapsto y^{2}x $& \multirow{2}{*}{(1\; 7\; 2\; 9\; 5\; 4\; 8\; 6\; 11\; 3\; 10)} & \multirow{2}{*}{$z\mapsto z^{3}$}& \multirow{2}{*}{(1 3 4)} \\
			& $y\mapsto xyx$ & & $y \mapsto z^{2}$ & & $y\mapsto yxy$ && & \\
\hline	
		\end{tabular}%
	}
	\caption{$\det_{H}$ and $\det_{K}$ corresponding to the cycle type $11+1$ of $\det_{H}$}\label{tbl14}
\end{table}


\newpage
	\begin{thebibliography}{00}
	
	\bibitem{bcm} J.N.S. Bidwell, M.J. Curran and D.J. McCaughan, Automorphisms of direct products of finite groups, \textit{Arch. Math.}, 86 (2006), 481--489.
		
		\bibitem{bc} J. N. S. Bidwell and M. J. Curran,	The automorphism group of a split metacyclic p-group, \textit{Arch. Math.}, 87 (2006), 488--497.
		
		\bibitem{jb} J.N.S. Bidwell, Automorphisms of direct products of finite groups II, \textit{Arch. Math.}, 91 (2006), 111--121.
		
		\bibitem{mb} M. Brescia, A determinant for automorphisms of groups, \textit{Commun. Algebra}, 53(6) (2025), 2484--2509. 
		
		\bibitem{cur} M. J. Curran, Automorphisms of semidirect products, \textit{Math. Proc. R. Ir. Acad.}, 108 (2008), 205--210.
		
		\bibitem{jd} J. Dietz, Automorphisms of products of groups, In CM. Campbell, M.R. Quick, E.F. Robertson, G.C. Smith (eds), Groups St. Andrews, LMS Lecture Notes Series 339 (2005), 28--305.
		
		\bibitem{af} A. Firat and C. Sinan, Knit products of some groups and their applications, \textit{Rend. Semin. Mat. Univ. Padova}, 121 (2009), 1--11.
		
		\bibitem{hsu} 	N.C. Hsu, The group of automorphisms of the holomorph of a group, \textit{Pacific J. Math.}, 11 (1961), 999--1012.
		
		\bibitem{zapfix}   V. Kakkar and R. Lal, Automorphisms of Zappa-Sz\'{e}p product fixing a subgroup, \textit{Acta Univ. Sapientiae, Mathematica}, 15(2) (2023), 288--303.
		
		\bibitem{zap} R. Lal and V. Kakkar, Automorphisms of Zappa-Sz\'{e}p product, \textit{Adv. Group Theory Appl.}, 14 (2022), 95--135.
		
		\bibitem{detsem} R. Lal, A. Choudhary and V. Kakkar, Determinant for automorphisms of semidirect product of groups, \textit{	arXiv:2507.18456}, (2025). https://doi.org/10.48550/arXiv.2507.18456
		
	
	\end{thebibliography}
\end{document}